\documentclass[a4paper,twoside,11pt,reqno]{article}

\usepackage[plain]{fullpage} 
\usepackage{amsmath}
\usepackage{amssymb}
\usepackage{latexsym}
\usepackage{amsxtra}
\usepackage{amscd}
\usepackage{amsthm}
\usepackage{mathabx}
\usepackage{enumerate}
\input xy
\xyoption{all}
\entrymodifiers={+!!<0pt,\fontdimen22\textfont2>}

\usepackage{graphics}
\usepackage{epic}

\usepackage{mathrsfs, amsfonts, latexsym}

\usepackage{units}

\usepackage{color}

\makeatletter
\def\author#1{\gdef\autrun{\def\and{\unskip, }#1}\gdef\@author{#1}}

\makeatother

\newtheorem{thm}{Theorem}[section]
\newtheorem{cor}[thm]{Corollary}
\newtheorem{lem}[thm]{Lemma}
\newtheorem{prop}[thm]{Proposition}

\newtheorem{mainthm}[thm]{Main Theorem}

\theoremstyle{definition}
\newtheorem{defin}[thm]{Definition}
\newtheorem{rem}[thm]{Remark}

\numberwithin{equation}{section}

{\newtheoremstyle{sectstyle}
  {6pt}
  {6pt}
  {}
  {}
  {\bf }
  {}
  {5pt}
  {}

\theoremstyle{sectstyle}

\newtheorem{sect}[thm]{} 
}

\renewcommand{\em}{\sl}

\DeclareMathOperator*{\colim}{colim}

\makeatletter
\renewcommand{\subsection}{\@startsection{subsection}{2}%
{\z@}{-3.25ex plus -1ex minus-.2ex}{-1em}{\bf}} \makeatother

\newcommand{\LeftEqNo}{\let\veqno\leqno}

\numberwithin{equation}{section}
\numberwithin{thm}{section}
\theoremstyle{plain}

\title{Canonical Lifts of Cycle Classes of Sections}
\author{Johannes Schmidt\thanks{Supported by DFG-Forschergruppe 1920 "Symmetrie, Geometrie und Arithmetik", Heidelberg--Darmstadt}}

\date{\vspace{-5ex}}

\begin{document}

\maketitle

\begin{quotation} 
\noindent \small {\bf Abstract.}
We present the construction of canonical lifts of $\ell$-adic cycle classes of sections of $p$-adic projective anabelian curves to the cohomology of arbitrary proper, regular, flat models.
This answers a question of Esnault and Wittenberg.
\end{quotation}


\section{Introduction}

\noindent {\bf Background.}
Let $X/k$ be a geometrically connected smooth projective curve of positive genus.
Fix a separable closure $k^{\rm s}/k$ with absolute Galois group $G_k$.
Choose a compatible base-point $\bar{x}_0$ and consider the \'{e}tale fundamental group sequence
\begin{equation*}\LeftEqNo\tag*{$\pi_1(X/k)$:}
 \xymatrix{
  \boldsymbol{1} \ar[r] &
  \pi_1(X\otimes_kk^{\rm s},\bar{x}_0) \ar[r] &
  \pi_1(X,\bar{x}_0) \ar[r] &
  G_k \ar[r] &
  \boldsymbol{1}
 }.
\end{equation*}
For a $k$-rational point $x$ let $\bar{x}$ be the geometric point over $x$ compatible with $k^{\rm s}/k$.
There is an isomorphism between $\pi_1(X,\bar{x})$ and $\pi_1(X,\bar{x}_0)$ over $G_k$, well defined up to conjugation by elements in the geometric fundamental group $\pi_1(X\otimes_kk^{\rm s},\bar{x}_0)$.
In particular, the $k$-point $x$ defines a $\pi_1(X\otimes_kk^{\rm s},\bar{x}_0)$-conjugacy class $[s_x]$ of sections of $\pi_1(X/k)$.
In his letter to Faltings, Grothendieck conjectured that the corresponding assignment $x \mapsto [s_x]$ is a bijection between the set of $k$ rational points of $X$ and the set of $\pi_1(X\otimes_kk^{\rm s},\bar{x}_0)$-conjugacy classes of sections of $\pi_1(X/k)$ for $k/\mathbb{Q}$ finitely generated and $X$ of genus $\geq 2$.
This conjecture is known under the name \emph{Section conjecture of anabelian geometry}.
The analogue conjecture over a $p$-adic field $k$ is known as \emph{$p$-adic Section conjecture}.
Both conjectures are still wide open.
In both cases, the essential question is whether the existence of a section of $\pi_1(X/k)$ implies the existence of a $k$-rational point in $X$.

\bigskip
\noindent {\bf The cycle class of a section.}
A linearized form of the above questions is whether the existence of a section implies at least the existence of an algebraic $0$-cycle of degree $1$ in $X$.
An explicit candidate for such an algebraic cycle is the cycle class ${\rm cl}_s$ of a section $s$, a class in ${\rm H}^2(X_{\rm \acute{e}t},\hat{\mathbb{Z}}(1))$ of degree $1$ defined in a similar way as the cycle class of a rational point (cf.~\cite[Chap.~6]{Stix12}).
In particular, if $x$ is a rational point of $X$, ${\rm cl}_{s_x}$ is just the cycle class ${\rm cl}_x$ of $x$.
An early appearance of this cycle class in the context of anabelian geometry was Parshin's observation in his proof of the geometric Mordell theorem, that cycle classes of sections of certain generically smooth relative curves only depend on the induced sections of certain fundamental group sequences (cf.~\cite{Parshin90} proof of Prop.~1).
Further, algebraicity results for cycle classes of certain sections were an important step in Mochizuki's celebrated proof of Grothendieck's anabelian conjectures for dominant morphisms to hyperbolic curves in \cite{Mochizuki99}.

From now on, let us focus on the $p$-adic version of the linearized form of the section conjecture.
So  fix a $p$-adic field $k$.
Again, the cycle classes of sections appear prominently:
Using model theoretic arguments, Koenigsmann showed in \cite{Koenigsmann05} that any birational section (i.e., a sections of $G_{K(X)}\rightarrow G_k$) is given by a $k$-rational point.
The algebraicity of ${\rm cl}_s$ was proven by \cite{EsnaultWittenberg09} for such birational sections $s$, thus giving a purely algebraic proof at least for the linearized version of this birational $p$-adic section conjecture.
For a general section $s$ of $\pi_1(X/k)$ promising results towards this linearized question have been achieved, too:
In \cite{Stix10}, Stix showed that there is at least an algebraic $0$-cycle of degree a power of $p$ (see op.~cit.~Thm.~15).
Further, Esnault and Wittenberg showed that ${\rm cl}_s$ is an algebraic cycle in $\ell$-adic cohomology ${\rm H}^2(X_{\rm \acute{e}t},\mathbb{Z}_\ell(1))$, at least for $\ell \neq p$ (see \cite[Cor.~3.4]{EsnaultWittenberg09}).

\bigskip
\noindent {\bf The problem of finding canonical lifts.}
Let $\mathfrak{o}$ be the ring of integers of $k$ and $\mathfrak{X}/\mathfrak{o}$ a regular proper flat model of $X$.
It follows from the $\ell$-adic algebraicity of ${\rm cl}_s$ that the $\ell$-adic cycle class admits a non-canonical lift to ${\rm H}^2(\mathfrak{X}_{\rm \acute{e}t},\mathbb{Z}_\ell(1))$.
However, it is still an open question (see \cite[Rem.\ 3.7 (iii)]{EsnaultWittenberg09}) if there is even a \emph{canonical} lift to the cohomology of $\mathfrak{X}$, i.e., a lift that can be defined \emph{independently} of arbitrary choices. 
The existence of such a canonical lift is predicted by the $p$-adic section conjecture or even the ``full'' algebraicity of ${\rm cl}_s$ (i.e., ${\rm cl}_s$ is the Chern class of a divisor in ${\rm H}^2(X_{\rm \acute{e}t},\hat{\mathbb{Z}}(1))$):
Indeed, under both assumptions ${\rm cl}_s$ is the cycle class of a \emph{unique} divisor class $[D]$ of $X$ (use \cite[Ch.~I~Lem.~3.3]{Milne04}), so the cycle class of the closure of $D$ in $\mathfrak{X}$ is a canonical lift.
The aim of this paper is to give an \emph{unconditional} construction of such a canonical lift (see Thm.~\ref{thm: canonical lift of the cycle class to the model}, below):

\begin{mainthm}\label{mainthm A}
Let $X/k$ be a geometrically connected smooth projective curve of positive genus over a $p$-adic field $k$ and $\mathfrak{X}/\mathfrak{o}$ a regular proper flat model.
Then for any prime $\ell \neq p$ and any section $s$ of $\pi_1(X/k)$, the induced $\ell$-adic cycle class ${\rm cl}_s$ admits a canonical lift ${\rm cl}_s^{\mathfrak{X}}$ to ${\rm H}^2(\mathfrak{X},\mathbb{Z}_\ell(1))$.
Further, the class ${\rm cl}_s^{\mathfrak{X}}$ is natural in the pair $(\mathfrak{X}/\mathfrak{o},s)$.
\end{mainthm}

Previously, such a canonical lift could only be given if the reduced special fibre $Y$ does not admit any rational components.
The reason is that only in this case the model $\mathfrak{X}$ is still a $K(\pi,1)$-space, i.e. its \'{e}tale cohomology of locally constant constructible coefficients is given by the cohomology of its fundamental group (cf.~Lem.~\ref{lem: homotopy type of a curve}).
Without the $K(\pi,1)$-property, it is much harder to link group theoretic construction involving the fundamental group sequence to the world of \'{e}tale cohomology in which lifts of the cycle class live.
The solution is to think about sections of $\pi_1(X/k)$ as homotopy rational (resp.~ homotopy fixed) points of $X/k$, i.e., splittings of the structural map $\widehat{\rm Et}(X) \rightarrow \widehat{\rm Et}(k) \simeq BG_{k}$ of \'{e}tale homotopy types in the homotopy category of simplicial profinite sets over the classifying space $BG_{k}$ (cf.\ \cite{Quick08}).
We then relate these to homotopy rational points of the special fibre $Y/\mathbb{F}$, giving us a grasp on the \'{e}tale cohomology of the model.
The idea to attack questions around the section conjectures using \'{e}tale homotopy theory has been around for several years already, among others e.g.~by Harpaz, P\'{a}l, Quick, Schlank, A.~Schmidt, Stix or Wickelgren (cf.~\cite[Sect.~2.6]{Stix12}).
Let us also mention A.~Schmidt and Stix's work towards anabelian geometry based on \'{e}tale homotopy theory in \cite{SchmidtStix16}.
Unfortunately, with respect to the section conjectures, outside of obstruction theory, only limited progress has been made made so far.
Further, it seems the homotopy theory of a model of a $p$-adic curve has not been included into this attempts up until now.
However, exactly the latter turned out to be very useful in the construction of the canonical lift of ${\rm cl}_s$. 

\bigskip
\noindent {\bf The construction of a canonical lift.}
The construction of ${\rm cl}_s$ can be rephrased using the pullback map $\mathbb{R}\Gamma(X_{\rm \acute{e}t},\Lambda) \rightarrow \mathbb{R}\Gamma(k_{\rm \acute{e}t},\Lambda)$ on \'{e}tale cohomology induced by the section $s$ (cf.~\ref{para: definition of the cycle class, reformulation II}).
Here it is crucial that $X$ is a $K(\pi,1)$-space.
The idea for the construction of ${\rm cl}_s^{\mathfrak{X}}$ is to specialize the section $s$ to a section $\bar{s}$ of $\pi_1(Y/\mathbb{F})$ and mimic the construction of ${\rm cl}_s$ for $\bar{s}$ (at least, after geometric pro-$\ell$ completion, such a specialization does exist -- cf.~Lem.~\ref{lem: unramified on geometric l completion}).
A priori this specialized section $\bar{s}$ may no longer induce a pullback map on \'{e}tale cohomology since $Y$ may contain rational components.
We solve this problem by showing that sections of $\pi_1(Y/\mathbb{F})$ canonically correspond to homotopy rational points of $Y/\mathbb{F}$ (cf.~Cor.~\ref{cor: unramified on geometric l completion and ho rat pts}).
From such a homotopy rational point we get the needed pullback map, allowing us the construction of a cycle class ${\rm cl}_{\bar{s}}$ in ${\rm H}^2(\mathfrak{X},\mathbb{Z}_\ell(n))$ (cf.~Def.~\ref{def: cycle class of ho rat points of special fibre}).
The desired lifting property of ${\rm cl}_{\bar{s}}$ will follow from the compatibility of these two pullback maps via the specialization morphism of \'{e}tale cohomology (cf.~Thm.~\ref{thm: specialized homotopy rational point pro l}).
This compatibility in turn will follow from a more explicit description of the specialized homotopy rational point $\bar{s}$ after an unramified base extension $k^\prime/k$ (with residue field extension $\mathbb{F}^\prime/\mathbb{F}$), using the existence of an ``algebraic'' section of the canonical map from $\widehat{\rm Et}(Y\otimes_{\mathbb{F}}\mathbb{F}^\prime)$ into the classifying space of its fundamental groupoid (cf.~Prop.~\ref{prop: specialized homotopy rational point pro l}).
\\
Although \'{e}tale homotopy theory is crucial in the construction and canonicity of the lift, at least the lifted class itself can be written down purely in terms of \'{e}tale cohomology and cohomology of profinite groups (cf.~Rem.~\ref{rem: construction of the cycle class without homotopy}).

\bigskip
\noindent {\bf Notation.}
In the following, $k$ always denotes a $p$-adic field with ring of integers $\mathfrak{o}$ and residue field $\mathbb{F}$.
If $K$ is any field, we denote a fixed separable closure by $K^{\rm s}$ and the corresponding absolute Galois group by $G_K$.
Denote by $\hat{\mathcal{S}}_{(*)}$ the category of (pointed) simplicial profinite sets together with the model structure of \cite{Quick08}.
By the $n^{\rm th}$ degree of a simplicial profinite set $\mathcal{X}$ we mean the profinite set $\mathcal{X}_n$.
If $\mathcal{X}$ is given by a pro-system $\{\mathcal{X}_i\}_{i\in I}$ of simplicial finite sets, then by the $i^{\rm th}$ level of $\mathcal{X}$ we mean the simplicial set $\mathcal{X}_i$.
For $X$ a scheme together with a geometric point $\bar{x}$, $\pi_1(X,\bar{x})$ denotes its profinite \'{e}tale fundamental group.
Mostly we will skip the base-point in our notation.
Denote by $\widehat{\rm Et}(X)$ its profinite \'{e}tale homotopy type in $\hat{\mathcal{S}}_{(*)}$.
For a simplicial profinite set $\mathcal{Y}$ with torsion local system $\Lambda$, write $C^\bullet(\mathcal{Y},\Lambda)$ for its cohomology cochains (see \cite[Sect.\ 2.2]{Quick08}).
If $\mathcal{Y}$ is the homotopy type $\widehat{\rm Et}(X)$ of a scheme $X$, then $C^\bullet(\widehat{\rm Et}(X),\Lambda)$ is quasi-isomorphic to $\mathbb{R}\Gamma(X_{\rm \acute{e}t},\Lambda)$ (see \cite[Sect.\ 3.1]{Quick08}).
If $\mathcal{Y}$ is the classifying space $BG$ of a profinite group $G$, then $C^\bullet(BG,\Lambda)$ is quasi-isomorphic to $\mathbb{R}\Gamma(G,\Lambda)$.
We will always use continuous \'{e}tale cohomology in the sense of \cite{Jannsen88}.
Finally, an index in brackets ``$A_{(i)}$'' usually means ``$A_i$ respectively $A$'', e.g.~``$f_{({\rm fet}),\ast}$'' means the finite-\'{e}tale resp.~\'{e}tale direct image $f_{{\rm fet},\ast}$ resp.~$f_{\ast}$.
There is one exception: $(-)_{(\ell)}^\wedge$ or $\pi_1^{(\ell)}$ always means the geometrically pro-$\ell$ completion (cf.~\ref{para: geometric pro l completion of a group} and \ref{para: geometric pro l completion of fundamental groups}, below).

\bigskip
\noindent {\bf Acknowledgments.}
I would like to thank H\'{e}l\`{e}ne Esnault, Armin Holschbach, Gereon Quick and Jakob Stix for helpful comments, discussions and suggestions.
I would especially like to thank the anonymous referees for the many helpful remarks and suggestions, in particular for suggesting the usage of obstruction theory in the construction of the quasi-specialized homotopy rational point.
This input allowed to delete an additional assumption in a previous version of the main theorem, excluding finitely many primes $\ell$.
This research project was partly supported by DFG-Forschergruppe 1920 "Symmetrie, Geometrie und Arithmetik", Heidelberg--Darmstadt.

\section{Preliminaries}
We will first recall and develop some techniques of \'{e}tale homotopy theory in Quick's setting of simplicial profinite sets.

\subsection*{Homotopical algebra.}
We will work in the following model categories:

\begin{sect}\label{para: homotopy categories}
Let $\hat{\mathcal{S}}$ be the category of simplicial profinite sets together with the model structures of \cite{Quick08}.
We will call objects in $\hat{\mathcal{S}}$ profinite spaces.
For a profinite group $G$, let $BG$ be its profinite classifying space and $\hat{\mathcal{S}}_{/ BG}$ the category of simplicial profinite sets over $BG$ together with the induced model structure.
A simplicial profinite $G$-set is a simplicial profinite set together with a degreewise continuous $G$-action.
Let $\hat{\mathcal{S}}_G$ be the resulting category together with the model structure of \cite{Quick10}.
Again, we will call objects in $\hat{\mathcal{S}}_G$ profinite $G$-spaces.
By \cite[Cor.\ 2.11]{Quick10}, $\hat{\mathcal{S}}_{/ BG}$ is Quillen equivalent to $\hat{\mathcal{S}}_G$ via the base change functor $\mathcal{X} \mapsto \mathcal{X} \times_{BG}EG$, where $EG \rightarrow BG$ is the universal $G$-torsor.
The corresponding Quillen-right-adjoint is given by the Borel-construction $(-) \times_G EG$.
Here, for a profinite $G$-space $\mathcal{Y}$, the Borel-construction $\mathcal{Y} \times_G EG$ is given as the product space $\mathcal{Y} \times EG$ modulo its diagonal $G$-action.
Under this equivalence, maps $BG\rightarrow \mathcal{X}$ in $\mathcal{H}(\hat{\mathcal{S}}_{/ BG})$ correspond to homotopy fixed points of $\mathcal{X} \times_{BG}EG$: 
\end{sect}

\begin{defin}\label{def: ho fixed points}
A \emph{homotopy fixed point} of a profinite $G$-space $\mathcal{X}$ is a morphism ${\rm pt} \rightarrow \mathcal{X}$ in $\mathcal{H}(\hat{\mathcal{S}}_G)$.
\end{defin}

\begin{rem}\label{rem: ho fixed points def}
Our definition of homotopy fixed points differs from the usage in \cite{Quick10} (in the sense of Sullivan):
Since $G$ acts freely on $EG$, $EG \rightarrow {\rm pt}$ is a cofibrant replacement in $\hat{\mathcal{S}}_G$.
Let $\mathcal{X} \hookrightarrow \mathcal{X}^\prime$ be a (functorial) fibrant replacement.
Then the set of homotopy fixed points \`{a} la Def.~\ref{def: ho fixed points} is given as
\begin{equation*}
 [{\rm pt},\mathcal{X}]_{\hat{\mathcal{S}}_G} = \pi_0({\rm map}_{\hat{\mathcal{S}}}(EG,\mathcal{X}^\prime)^G).
\end{equation*}
The mapping space $\mathcal{X}^{hG} = {\rm map}_{\hat{\mathcal{S}}}(EG,\mathcal{X}^\prime)^G$ is Quick's homotopy fixed point space, defined and studied in \cite{Quick10}.
In particular, a homotopy fixed point in the sense of \cite{Quick10} is a $0$-simplex of $\mathcal{X}^{hG}$, whereas a homotopy fixed point in the sense of Def.~\ref{def: ho fixed points} is the connected component of such a $0$-simplex.
\end{rem}

\begin{sect}\label{para: ho fixed points}
In general, the homotopy fixed point space $\mathcal{X}^{hG}$ is difficult to describe.
At least, by \cite[Thm.~2.16]{Quick10}, there is a Bousfield-Kan type descent spectral sequence (with differentials in the usual ``cohomological'' directions)
\begin{equation}\label{eq: descent spectral sequence}
 E_2^{p,q} = {\rm H}^p(G,\pi_{-q}(\mathcal{X})) \Rightarrow \pi_{-(p+q)}(\mathcal{X}^{hG}).
\end{equation}
\end{sect}

Applying Bousfield and Kan's connectivity lemma \cite[Chap.~IX 5.1]{BousfieldKan} to the spectral sequence in \ref{para: ho fixed points}, one can prove: 

\begin{lem}\label{lem: injection of hofixed points}
Let $G$ be a profinite group of cohomological dimension $\leq n$.
Let $f\colon \mathcal{X} \rightarrow \mathcal{Y}$ be an $(n+1)$-equivalence in $\hat{\mathcal{S}}_G$ (i.e., $\pi_q(f)$ is an isomorphism for all $q \leq n$ and an epimorphism for $q = n+1$).
Then $f$ induces an injection
\begin{equation*}
 [EG,\mathcal{X}]_{\hat{\mathcal{S}}_G} =
 \xymatrix{
  \pi_0(\mathcal{X}^{hG}) \ar@{^(->}[r] &
  \pi_0(\mathcal{Y}^{hG})
 }
 = [EG,\mathcal{Y}]_{\hat{\mathcal{S}}_G}.
\end{equation*}
\end{lem}

\begin{proof}
We may assume that $\mathcal{Y}$ is fibrant and $f$ is a fibration in $\hat{\mathcal{S}}_G$.
Further, we may assume that $\mathcal{X}^{hG}$ is non-empty.
Say, $s\colon EG \rightarrow \mathcal{X}$ is a model of a homotopy fixed point and let $r$ be $f \circ s$.
The fibre $\mathcal{F}_r := \mathcal{X} \times_{\mathcal{Y}} EG$ comes equipped with a fibration into $EG$, hence is fibrant in $\hat{\mathcal{S}}_G$, too.
Taking limits (i.e., forgetting the topology in \cite{Quick10}) resp.\ simplicial mapping spaces ${\rm map}_{\hat{\mathcal{S}}}(EG,-)^G$ gives us a homotopy fibre sequence
\begin{equation*}
 \xymatrix{
  \lim\mathcal{F}_r \ar[r] &
  \lim\mathcal{X} \ar[r] &
  \lim\mathcal{Y}
 }
\end{equation*}
resp.\
\begin{equation*}
 \xymatrix{
  \mathcal{F}_r^{hG} \ar[r] &
  \mathcal{X}^{hG} \ar[r] &
  \mathcal{Y}^{hG}
 }
\end{equation*}
in $\underline{\rm SSets}_\ast$ (pointed by the neutral element in $G$).
By \cite[Lem. 2.9]{Quick13}, the limit of $f$ is an $n$-equivalence of simplicial sets. 
So, again by loc.~cit., the first fibre sequence implies the $n$-connectedness of $\mathcal{F}_r$.
Using the second homotopy fibre sequence, we get that the map of pointed sets $(\pi_0(\mathcal{X}^{hG}),s) \rightarrow (\pi_0(\mathcal{Y}^{hG}),r)$ has kernel $\pi_0(\mathcal{F}_r^{hG})$.
Bousfield and Kan's connectivity lemma applied to the descent spectral sequence (\ref{eq: descent spectral sequence}) for $\mathcal{F}_r$ implies that this kernel is trivial, since $\mathcal{F}_r$ is $n$-connected and $G$ has cohomological dimension $\leq n$.
Varying over all the homotopy fixed points of $\mathcal{X}$, we get the result.
\end{proof}

\begin{rem}\label{rem: obstruction theory}
Under the assumptions of Lem.\ \ref{lem: injection of hofixed points}, assume that $\mathcal{Y}^{hG}$ is non-empty.
Suppose $\mathcal{Y}$ is fibrant and $f$ a fibration in $\hat{\mathcal{S}}_G$ and let $r\colon EG \rightarrow \mathcal{X}$ be a model of a homotopy fixed point.
Again, $\mathcal{F}_r := \mathcal{X} \times_{\mathcal{Y}} EG$ is $n$-connected and in particular, ${\rm H}^{q+1}(G,\pi_q(\mathcal{F}_r))$ is trivial for all $q$.
Via a suitable obstruction theory (cf.\ \cite[6.1]{Bousfield89} or \cite[Chap.~4]{HarpazSchlank13}) it would follow that $\mathcal{F}_r^{hG}$ is non-empty, as well.
Thus, $\pi_0(\mathcal{X}^{hG}) \rightarrow \pi_0(\mathcal{Y}^{hG})$ would in fact be an isomorphism.
\end{rem}

We do not intend to develop such an obstruction theory at this place.
Instead, we will show the non-emptiness of $\mathcal{F}_r^{hG}$ in the case of $G = \hat{\mathbb{Z}}$ by hand:

\begin{lem}\label{lem: ho fixed points of simply connected spaces}
Let $\mathcal{F}$ be a simply connected space in $\hat{\mathcal{S}}_{\hat{\mathbb{Z}}}$.
Then $\mathcal{F}^{h\hat{\mathbb{Z}}}$ is non-empty.
\end{lem}

\begin{proof}
Let $\mathcal{E}/B\hat{\mathbb{Z}}$ be fibrant in $\hat{\mathcal{S}}_{/ B\hat{\mathbb{Z}}}$ with $\mathcal{F} \simeq \mathcal{E}\times_{B\hat{\mathbb{Z}}}E\hat{\mathbb{Z}}$.
The structural map $\mathcal{E} \rightarrow B\hat{\mathbb{Z}}$ induces an isomorphism $\pi_1(\mathcal{E},e) = \hat{\mathbb{Z}}$ for any point $e\in \mathcal{E}_0$.
By \cite[Prop.~2.17]{Quick08}, $\pi_1(\mathcal{E},e) = \pi_0(\Omega(\mathcal{E},e))$.
Let ${\rm S}^1 = \Delta^1/\partial \Delta^1 \rightarrow \mathcal{E}$ be the homotopy class corresponding to the generator $1\in \hat{\mathbb{Z}}$.
Since ${\rm S}^1 \simeq B\hat{\mathbb{Z}}$ in $\hat{\mathcal{S}}$, this induces a homotopy class $B\hat{\mathbb{Z}} \rightarrow \mathcal{E}$ in $\mathcal{H}(\hat{\mathcal{S}}_{/ B\hat{\mathbb{Z}}})$, i.e., a homotopy fixed point of $\mathcal{F}$.
\end{proof}

As in Rem.~\ref{rem: obstruction theory}, Lem.~\ref{lem: ho fixed points of simply connected spaces} implies:

\begin{cor}\label{cor: bijection of hofixed points}
Let $f\colon \mathcal{X} \rightarrow \mathcal{Y}$ be a $2$-equivalence in $\hat{\mathcal{S}}_{\hat{\mathbb{Z}}}$.
Then $f$ induces a bijection
\begin{equation*}
 [EG,\mathcal{X}]_{\hat{\mathcal{S}}_{\hat{\mathbb{Z}}}} =
 \xymatrix{
  \pi_0(\mathcal{X}^{h{\hat{\mathbb{Z}}}}) \ar[r]^-\sim &
  \pi_0(\mathcal{Y}^{h{\hat{\mathbb{Z}}}})
 }
 = [EG,\mathcal{Y}]_{\hat{\mathcal{S}}_{\hat{\mathbb{Z}}}}.
\end{equation*}
\end{cor}

\begin{sect}\label{para: group epis induce fibration between classifying spaces}
Let $p\colon \pi \twoheadrightarrow G$ be an epimorphism of profinite groups.
Then the induced map $B\pi \rightarrow BG$ between the classifying spaces is a fibration:
To see this, first note that the cofibrations in $\hat{\mathcal{S}}$ are precisely the degreewise monomorphisms (cf.~\cite[Thm.~2.12]{Quick08}). 
Using the adjunction between the profinite groupoid $\Pi(-)$ and $B(-)$, we have to solve the lifting problem
\begin{equation*}
 \xymatrix{
  \Delta \ar[r] \ar[d] & 
  \pi \ar@{->>}[d]
 \\
  \Gamma \ar[r] \ar@{-->}[ru] &
  G
 }
\end{equation*}
for $\Gamma$ a connected profinite groupoid and $\Delta \subseteq \Gamma$ a connected full profinite subgroupoid.
Using the usual exactness properties of pro-categories, 
this can be solved as in the discrete case
via the standard argument applying Zorn's Lemma.
\end{sect}

\begin{sect}\label{para: sect and vs ho rat points of classifying spaces}
Let $p\colon \pi \twoheadrightarrow G$ be an epimorphism of profinite groups with kernel $\pi^\prime\unlhd \pi$.
Every object is cofibrant in $\hat{\mathcal{S}}$ and $B\pi \rightarrow BG$ is fibrant by \ref{para: group epis induce fibration between classifying spaces}.
Thus, $[BG,B\pi]_{\hat{\mathcal{S}}_{/ BG}}$ is given as ${\rm Hom}_{\hat{\mathcal{S}}_{/ BG}}(BG,B\pi)$ modulo homotopy equivalences over $BG$ with respect to the standard cylinder object $BG \otimes \Delta^1$.
Such homotopies between maps $B(s_0)$ and $B(s_1)$ for sections $s_i$ of $p$ correspond precisely to conjugation of these sections via elements of $\pi^\prime$.
In particular, $B(-)$ and $\pi_1(-,\ast)$ give canonical identifications between $[BG,B\pi]_{\hat{\mathcal{S}}_{/ BG}}$ and the set of $\pi^\prime$-conjugacy classes of section of $p$.
\end{sect}

\begin{sect}\label{para: sections vs ho rat points}
Let $\mathcal{X} /BG$ be a connected simplicial profinite set in $\hat{\mathcal{S}}_{/ BG}$ and assume $\pi_1(\mathcal{X}) \rightarrow G$ is an epimorphism.
Then any map $BG \rightarrow \mathcal{X}$ in $\mathcal{H}(\hat{\mathcal{S}}_{/ BG})$ defines a $\pi_1(\mathcal{X}\times_{BG}EG)$-conjugacy class of splittings of the fundamental group sequence
\begin{equation*}
 \xymatrix{
  \boldsymbol{1} \ar[r] &
  \pi_1(\mathcal{X}\times_{BG}EG) \ar[r] &
  \pi_1(\mathcal{X}) \ar[r] &
  G \ar[r] &
  \boldsymbol{1}
 }
\end{equation*}
of $\mathcal{X} /BG$.
Conversely, if the underlying simplicial profinite set $\mathcal{X}$ of $\mathcal{X} /BG$ is a $K(\pi,1)$ (i.e., the canonical map $\mathcal{X} \rightarrow B\Pi(\mathcal{X})$ into the classifying space of the profinite fundamental groupoid is a weak equivalence), it follows from \ref{para: sect and vs ho rat points of classifying spaces} that sections of the above fundamental group sequence of $\mathcal{X} /BG$ modulo conjugation correspond to maps $BG \rightarrow \mathcal{X}$ in $\mathcal{H}(\hat{\mathcal{S}}_{/ BG})$.
\end{sect}

Combining \ref{para: sections vs ho rat points} with Lem.\ \ref{lem: injection of hofixed points} and Cor.~\ref{cor: bijection of hofixed points}, we get:

\begin{lem}\label{lem: sections vs ho rat points}
Let $G$ be a profinite group of cohomological dimension $\leq 1$ and $\mathcal{X} /BG$ a connected simplicial profinite set in $\hat{\mathcal{S}}_{/ BG}$ s.t.\ $\pi_1(\mathcal{X}) \rightarrow G$ is an epimorphism.
Assume\footnote{Using a suitable obstruction theory, we could delete this assumption (cf.~Rem.~\ref{rem: obstruction theory}).} either that the canonical map $\mathcal{X}\times_{BG}EG \rightarrow B\Pi(\mathcal{X}\times_{BG}EG)$ admits a section in $\mathcal{H}(\hat{\mathcal{S}}_G)$ or that $G=\hat{\mathbb{Z}}$.
Then we get a canonical identification between the set of $\pi_1(\mathcal{X}\times_{BG}EG)$-conjugacy classes of sections of $\pi_1(\mathcal{X}) \rightarrow G$ and $[BG,\mathcal{X}]_{\hat{\mathcal{S}}_{/ BG}} \simeq [EG,\mathcal{X}\times_{BG}EG]_{\hat{\mathcal{S}}_G}$.  
\end{lem}

\begin{proof}
Indeed, the canonical map $\mathcal{X}\times_{BG}EG \rightarrow B\Pi(\mathcal{X}\times_{BG}EG)$ is a $2$-equivalence, so it induces an injection on the respective sets of homotopy fixed points by Lem.\ \ref{lem: injection of hofixed points}.
A section of $\mathcal{X}\times_{BG}EG \rightarrow B\Pi(\mathcal{X}\times_{BG}EG)$ induces a section of the corresponding map on homotopy fixed points.
So the latter map is even bijective and the claim follows from \ref{para: sections vs ho rat points} applied to $B\Pi(\mathcal{X})$.
If $G = \hat{\mathbb{Z}}$, we just use Cor.~\ref{cor: bijection of hofixed points} instead of the section.
\end{proof}

\begin{sect}\label{para: subextension and restrictions 1}
Let $G \twoheadrightarrow \bar{G}$ be an epimorphism of profinite groups with kernel $G^\prime \unlhd G$.
Let $\mathcal{X}$ be a cofibrant and $\mathcal{Y}$ a fibrant profinite $G$-set.
Then $[\mathcal{X},\mathcal{Y}]_{\hat{\mathcal{S}}_G}$ is given as the set of connected components $\pi_0 ({\rm map}_{\hat{\mathcal{S}}}(\mathcal{X},\mathcal{Y})^G)$.
By \cite[Thm.~2.9 and Cor.~2.10]{Quick10}, ${\rm res}_{G^\prime}^G(\mathcal{X})$ is still cofibrant and ${\rm res}_{G^\prime}^G(\mathcal{Y})$ still fibrant in $\hat{\mathcal{S}}_{G^\prime}$.
In particular, the $\bar{G}$-action via conjugation on the mapping space ${\rm map}_{\hat{\mathcal{S}}}(\mathcal{X},\mathcal{Y})^{G^\prime}$ induces a canonical $\bar{G}$-action on
\begin{equation*}
 [{\rm res}_{G^\prime}^G(\mathcal{X}),{\rm res}_{G^\prime}^G(\mathcal{Y})]_{\hat{\mathcal{S}}_{G^\prime}} = \pi_0 ({\rm map}_{\hat{\mathcal{S}}}(\mathcal{X},\mathcal{Y})^{G^\prime}).
\end{equation*}
Further, the restriction ${\rm res}_{G^\prime}^G(-)\colon [\mathcal{X},\mathcal{Y}]_{\hat{\mathcal{S}}_G} \rightarrow [{\rm res}_{G^\prime}^G(\mathcal{X}),{\rm res}_{G^\prime}^G(\mathcal{Y})]_{\hat{\mathcal{S}}_{G^\prime}}$ factors through the invariants $[{\rm res}_{G^\prime}^G(\mathcal{X}),{\rm res}_{G^\prime}^G(\mathcal{Y})]_{\hat{\mathcal{S}}_{G^\prime}}^{\bar{G}}$.
Let $f\colon \mathcal{Y} \rightarrow \mathcal{Y}^\prime$ be a morphism between fibrant spaces in $\hat{\mathcal{S}}_G$.
Then ${\rm res}_{G^\prime}^G(f)$ induces a $\bar{G}$-equivariant map between homotopy classes $[{\rm res}_{G^\prime}^G(\mathcal{X}),{\rm res}_{G^\prime}^G(\mathcal{Y})]_{\hat{\mathcal{S}}_{G^\prime}} \rightarrow [{\rm res}_{G^\prime}^G(\mathcal{X}),{\rm res}_{G^\prime}^G(\mathcal{Y}^\prime)]_{\hat{\mathcal{S}}_{G^\prime}}$ (and similarly for a map $g\colon \mathcal{X}\rightarrow \mathcal{X}^\prime$ between cofibrant profinite $G$-spaces).
\\
Thus, for $\mathcal{X}$ and $f\colon \mathcal{Y}\rightarrow \mathcal{Y}^\prime$ in $\mathcal{H}(\hat{\mathcal{S}}_G)$ arbitrary, $[{\rm res}_{G^\prime}^G(\mathcal{X}),{\rm res}_{G^\prime}^G(\mathcal{Y})]_{\hat{\mathcal{S}}_{G^\prime}}$ carries a canonical $\bar{G}$-action, ${\rm res}_{G^\prime}^G(-)\colon [\mathcal{X},\mathcal{Y}]_{\hat{\mathcal{S}}_G} \rightarrow [{\rm res}_{G^\prime}^G(\mathcal{X}),{\rm res}_{G^\prime}^G(\mathcal{Y})]_{\hat{\mathcal{S}}_{G^\prime}}$ factors through the $\bar{G}$-invariants and $[{\rm res}_{G^\prime}^G(\mathcal{X}),{\rm res}_{G^\prime}^G(f)]_{\hat{\mathcal{S}}_{G^\prime}}$ is $\bar{G}$-equivariant (and similarly for a morphism $g\colon \mathcal{X}\rightarrow \mathcal{X}^\prime$ in $\mathcal{H}(\hat{\mathcal{S}}_G)$).
\end{sect}

\begin{sect}\label{para: subextension and restrictions 2}
Again, let $G \twoheadrightarrow \bar{G}$ be an epimorphism of profinite groups with kernel $G^\prime \unlhd G$.
Let $\mathcal{X} /BG$ be a connected simplicial profinite set in $\hat{\mathcal{S}}_{/ BG}$. 
The subgroup $G^\prime$ acts freely on ${\rm res}_{G^\prime}^G (\mathcal{X}\times_{BG}EG)$ and the projection ${\rm res}_{G^\prime}^G (\mathcal{X}\times_{BG}EG)\times EG^\prime \rightarrow {\rm res}_{G^\prime}^G (\mathcal{X}\times_{BG}EG)$ identifies the restrictions ${\rm res}_{\boldsymbol{1}}^{G^\prime}({\rm res}_{G^\prime}^G (\mathcal{X}\times_{BG}EG)  \times_{G^\prime} EG^\prime)$ and ${\rm res}_{\boldsymbol{1}}^{\bar{G}}(\mathcal{X}\times_{B\bar{G}}E\bar{G})$ in $\mathcal{H}(\hat{\mathcal{S}})$. 
In particular, the induced map $\pi_1({\rm res}_{G^\prime}^G (\mathcal{X}\times_{BG}EG) \times_{G^\prime} EG^\prime) \rightarrow G^\prime$ agrees with the projection map $\pi_1(\mathcal{X})\times_GG^\prime \rightarrow G^\prime$.
\end{sect}

Using \ref{para: subextension and restrictions 1} and \ref{para: subextension and restrictions 2}, we get the following $\bar{G}$-equivariant refinement of Lem.~\ref{lem: sections vs ho rat points}:

\begin{lem}\label{lem: sections vs ho rat points equivariant}
Let $G \twoheadrightarrow \bar{G}$ be an epimorphism of profinite groups with kernel $G^\prime \unlhd G$ of cohomological dimension $\leq 1$.
Let $\mathcal{X} /BG$ be a connected profinite space in $\hat{\mathcal{S}}_{/ BG}$ and $s$ a section of $\pi_1(\mathcal{X}) \rightarrow G$.
Suppose\footnote{Again, using a suitable obstruction theory, we could delete this assumption.} the canonical map ${\rm res}_{G^\prime}^G(\mathcal{X}\times_{BG}EG) \rightarrow B\Pi({\rm res}_{G^\prime}^G(\mathcal{X}\times_{BG}EG))$ admits a section in $\mathcal{H}(\hat{\mathcal{S}}_{G^\prime})$ or that $G^\prime=\hat{\mathbb{Z}}$.
Then the restricted section $s^\prime = s\vert_{G^\prime}$ of $\pi_1(\mathcal{X}) \times_GG^\prime \rightarrow G^\prime$ corresponds to a homotopy fixed point in the $\bar{G}$-invariants $[EG^\prime,{\rm res}_{G^\prime}^G(\mathcal{X}\times_{BG}EG)]_{\hat{\mathcal{S}}_{G^\prime}}^{\bar{G}}$.
\end{lem}

\subsection*{Homotopy rational points.}
We are mainly interested in profinite homotopy types of varieties over a field $K$:

\begin{sect}\label{para: profinite ho type}
Let $Y$ be a $K$-variety.
We work with a slightly modified version of Quick's profinite homotopy type (cf.~\cite[Sect.~3.1]{Quick08}):
Instead of hypercovers in the sense of \cite{Friedlander}, we work with hypercovers pointed by $K^{\rm s}$-points compatible with our choice of a separable closure $K^{\rm s}/K$.
In the mixed characteristics case over $\mathfrak{o}$, we use $k^{\rm s}$- and $\mathbb{F}^{\rm s}$-points compatible with $k^{\rm s}/k$ and $\mathbb{F}^{\rm s}/\mathbb{F}$.
It is not hard to see that this does not change the homotopy type.
The gain is that the \v{C}ech topological type of the spectrum of $K$ is the profinite classifying space $BG_K$ on the nose.
In particular, we get a canonical weak equivalence in $\hat{\mathcal{S}}$ from the profinite \'{e}tale homotopy type of the spectrum of $K$ to $BG_K$.
We define the profinite homotopy type $\widehat{\rm Et}(Y/K) \rightarrow BG_K$ of $Y/K$ as the resulting map in $\hat{\mathcal{S}}_{/ BG_K}$ induced by the structural map of $Y/K$.
Using \cite[Thm.~3.5 and Lem.~3.3]{Quick10}, we see that the underlying homotopy type of $\widehat{\rm Et}(Y/K)\times_{BG_K}EG_K$ corresponds to the homotopy type $\widehat{\rm Et}(Y\otimes_KK^{\rm s})$ in $\mathcal{H}(\hat{\mathcal{S}})_{G_K}$.
Similar arguments work for any Galois extension $K^\prime/K$ and $Y\otimes_KK^\prime$, too.
Further, $Y\otimes_KK^\prime \rightarrow Y$ induces a canonical weak equivalence $\widehat{\rm Et}(Y\otimes_KK^\prime/K^\prime)\times_{BG_{K^\prime}}EG_{K^\prime} \rightarrow {\rm res}_{G_{K^\prime}}^{G_K}(\widehat{\rm Et}(Y/K)\times_{BG_{K}}EG_{K})$ in $\hat{\mathcal{S}}_{G_{K^\prime}}$.
\end{sect}

\begin{sect}\label{para: ho rat points}
Each $K$-rational point of $Y$ defines a map $BG_K \rightarrow \widehat{\rm Et}(Y/K)$ in $\mathcal{H}(\hat{\mathcal{S}}_{/ BG_K})$, i.e., a homotopy fixed point of $\widehat{\rm Et}(Y/K) \times_{BG_K}EG_K$.
For any simplicial profinite set $\mathcal{X}/BG_K$, we therefore call any map $BG_K \rightarrow \mathcal{X}$ in $ \mathcal{H}(\hat{\mathcal{S}}_{/ BG_K})$ a {\bf homotopy rational point} of $\mathcal{X}$ over $K$.
By a homotopy rational (resp.~homotopy fixed point) of $Y/K$ we simply mean a homotopy rational point of $\widehat{\rm Et}(Y/K)$ (resp.~the induced homotopy fixed point of $\widehat{\rm Et}(Y/K) \times_{BG_K}EG_K$). 
\end{sect}

\begin{sect}\label{para: sections vs ho rat points of varieties}
Let $Y/K$ be a geometrically connected $K$-variety.
Then any homotopy rational point of $Y/K$ gives a conjugacy class of splittings of the fundamental group sequence $\pi_1(Y/K)$.
Conversely, assume $Y$ has the $K(\pi,1)$-property, i.e., its \'{e}tale cohomology of constructible locally constant coefficients is given by the cohomology of its finite \'{e}tale site.
It follows that $\widehat{\rm Et}(Y/K)$ is a $K(\pi,1)$-space in the above sense.
By \ref{para: sections vs ho rat points}, we get canonical identifications between the set of $\pi_1(Y\otimes_KK^{\rm s})$-conjugacy classes of the fundamental group sequence $\pi_1(Y/K)$, the set of homotopy rational points and the set of homotopy fixed points of $Y/K$ (see also \cite[Sect.~3.2]{Quick10}).
By \cite[Prop.\ A.4.1]{Stix02}, this in particular is the case for $Y$ any smooth curve over $K$ except for Brauer-Severi curves.
\end{sect}

Recall that a profinite group $\pi$ is $\ell$-good, if the pro-$\ell$ completion map $\pi \rightarrow \hat{\pi}_\ell$ induces isomorphisms ${\rm H}^q(\pi,\Lambda) \simeq {\rm H}^q(\hat{\pi}_\ell,\Lambda)$ for all finite $\ell$-torsion $\hat{\pi}_\ell$-modules $\Lambda$ and all these cohomology groups are finite.
For arbitrary curves we get:

\begin{lem}\label{lem: homotopy type of a curve}
Let $Y/K$ be a reduced, geometrically connected curve over field $K$.
\begin{enumerate}
 \item\label{lem: homotopy type of a curve Kpi1} $Y$ is a $K(\pi,1)$ if and only if it does not contain any rational projective components.
 \item\label{lem: homotopy type of a curve l good} If $K$ is separably closed, then $Y$ has an $\ell$-good fundamental group.
\end{enumerate}
\end{lem}

\begin{proof}
Since $Y$ is a $K(\pi,1)$ if and only if $Y\otimes_KK^{\rm s}$ is a $K(\pi,1)$, assume $K$ is separably closed.
By topological invariance of \'{e}tale and finite \'{e}tale cohomology, we may even assume $K$ is algebraically closed.
Weak-normalization is a universal homeomorphism by \cite[Thm.~4]{AndreottiBombieri} and coincides with semi-normalization over perfect fields in the curve case, i.e., we may assume that $Y$ is semi-normal.
Since $K$ is algebraically closed, this just means $Y$ has at most ordinary multiple points as singularities (use \cite[Sect.~I 7.2.2.1]{Kollar96}).
Let us first describe the fundamental group of $Y$.
Let $\pi\colon \tilde{Y} = \coprod_i\tilde{Y}_i \rightarrow Y$ be the normalization with connected components $\tilde{Y}_i$.
Then $\pi_1(Y)$ is the free product of the fundamental groups $\pi_1(\tilde{Y}_i)$ of normalized components and a finitely generated free profinite group $F$:
E.g., by adding additional projective lines we may (for this moment alone) assume that $Y$ is even semi-stable, hence we can apply \cite[Ex.~5.5]{Stix06}.
Since $F$ and the $\pi_1(\tilde{Y}_i)$ are $\ell$-good (for the latter, cf.~\cite[Prop.~A.4.1]{Stix02}), $\pi_1(Y)$ is $\ell$-good too and claim \ref{lem: homotopy type of a curve l good} holds.
\\
In order to prove claim \ref{lem: homotopy type of a curve Kpi1}, consider the canonical map $\gamma^\ast\colon {\rm H}^q(Y_{\rm fet},\Lambda) \rightarrow {\rm H}^q(Y,\Lambda)$ between finite-\'{e}tale and \'{e}tale cohomology with locally constant constructible coefficients $\Lambda$.
Using the following Lem.~\ref{lem:etale vs finite etale cohomology}, we see that $Y$ is a $K(\pi,1)$ if and only if all $\tilde{Y}_i$ are $K(\pi,1)$ and claim \ref{lem: homotopy type of a curve Kpi1} follows, again using \cite[Prop.~A.4.1]{Stix02}.
\end{proof}

\begin{lem}\label{lem:etale vs finite etale cohomology}
Let $Y/K$ be a reduced, geometrically connected curve over a field $K$.
Let $\pi\colon \tilde{Y} = \coprod_i\tilde{Y}_i \rightarrow Y$ be the normalization, $\tilde{Y}_i$ the irreducible components and $\mathcal{R}$ the set of indices $i$ with $\tilde{Y}_i$ a rational projective curve.
Let $f\colon Y \rightarrow {\rm Spec}(K)$ and $\tilde{f}_{(i)}\colon \tilde{Y}_{(i)} \rightarrow {\rm Spec}(K)$ be the structural morphisms and $\Lambda$ a locally constant constructible sheaf on $Y_{\rm et}$.
\begin{enumerate}
 \item\label{lem:etale vs finite etale cohomology fet} For $q\geq 2$, the canonical map $\mathbb{R}^q f_{{\rm fet},\ast}\Lambda \rightarrow \mathbb{R}^q \tilde{f}_{{\rm fet},\ast}\pi^\ast\Lambda$ is an isomorphism on finite-\'{e}tale cohomology groups in $\underline{\rm Mod}_{G_K}$.
 \item\label{lem:etale vs finite etale cohomology et} For $q\geq 2$, the canonical map $\mathbb{R}^q f_{\ast}\Lambda \rightarrow \mathbb{R}^q \tilde{f}_{\ast}\pi^\ast\Lambda$ is an isomorphism on \'{e}tale cohomology groups in $\underline{\rm Mod}_{G_K}$.
 \item\label{lem:etale vs finite etale cohomology triangle} The canonical morphism of sites $\gamma\colon Y_{\rm et} \rightarrow Y_{\rm fet}$ induces the following exact triangle in $\mathcal{D}^+(\underline{\rm Mod}_{G_K})$:
       \begin{equation*}
        \xymatrix{
         \mathbb{R} f_{{\rm fet},\ast}\Lambda \ar[r] &
         \mathbb{R} f_{\ast}\Lambda \ar[r] &
         \bigoplus_{i\in \mathcal{R}}  (\mathbb{R}^2 \tilde{f}_{i,\ast}\pi^\ast\Lambda)[-2] \ar[r] &
         \mathbb{R} f_{{\rm fet},\ast}\Lambda[1]
        }.
       \end{equation*}
\end{enumerate}
\end{lem}

\begin{proof}
As in the proof of Lem.~\ref{lem: homotopy type of a curve}, we may assume that $K$ is algebraically closed and $Y$ is semi-normal.
Again, $\pi_1(Y)$ is the free product of the fundamental groups $\pi_1(\tilde{Y}_i)$ of normalized components and a finitely generated free profinite group $F$.
For claims \ref{lem:etale vs finite etale cohomology fet} and \ref{lem:etale vs finite etale cohomology et}, we have to show that $\pi$ induces isomorphisms ${\rm H}^q(Y_{\rm fet},\Lambda)\rightarrow {\rm H}^q(\tilde{Y}_{\rm fet},\pi^\ast\Lambda)$ and ${\rm H}^q(Y,\Lambda)\rightarrow {\rm H}^q(\tilde{Y},\pi^\ast\Lambda)$ on finite-\'{e}tale and \'{e}tale cohomology groups of degree $q\geq 2$.
For the former, note that $F$ has cohomological dimension $\leq 1$ and group-cohomology in degrees $\geq 1$ of free products is the direct sum of the cohomology of the factors.
For the latter, use that $\pi$ is finite, so $\pi_\ast\pi^\ast\Lambda/\Lambda$ is a skyscraper sheaf.
\\
It remains to show claim \ref{lem:etale vs finite etale cohomology triangle}:
Let $C^\bullet$ be the cone of $\gamma^\ast\colon \mathbb{R} f_{{\rm fet},\ast} \rightarrow \mathbb{R} f_{\ast}\Lambda$.
In degrees $q = 0,1$, $\gamma$ induces isomorphisms on cohomology.
Further, by claims \ref{lem:etale vs finite etale cohomology fet} and \ref{lem:etale vs finite etale cohomology et},
\begin{equation*}
 \mathbb{R}^q f_{\ast}\Lambda \simeq \mathbb{R}^q f_{{\rm fet},\ast}\Lambda \oplus \bigoplus_{i\in \mathcal{R}} \mathbb{R}^q \tilde{f}_{i,\ast}\pi^\ast\Lambda
\end{equation*}
holds for $q \geq 2$ and $\gamma^\ast$ corresponds to the embedding $\mathbb{R}^q f_{{\rm fet},\ast}\Lambda \hookrightarrow \mathbb{R}^q f_{\ast}\Lambda$:
Indeed, for $i\in \mathcal{R}$ the component $\tilde{Y}_i \simeq \mathbb{P}_K^1$ is simply connected, while $\tilde{Y}_j$ are $K(\pi,1)$ for the remaining $j \notin \mathcal{R}$.
Using the long exact sequence on cohomology and ${\rm cd}(\mathbb{P}_K^1) = 2$, we see that $C^\bullet$ has cohomology only in degree $2$ and ${\rm H}^2C^\bullet \cong \bigoplus_{i\in \mathcal{R}}  \mathbb{R}^2 \tilde{f}_{i,\ast}\pi^\ast\Lambda$.
\end{proof}

\begin{rem}\label{rem: homotopy type of a curve}
A more precise description of the homotopy type of a curve is given in \cite[Thm.~2.4]{JSchmidt16}.
In particular, under suitable rationality-conditions on the singularities of $Y$ and if $K$ has cohomological dimension $\leq 1$, its homotopy type can be described as its $K(\pi,1)$-part glued to the homotopy types of its rational projective components, mirroring the exact triangle in Lem.~\ref{lem:etale vs finite etale cohomology}.
\end{rem}

\subsection*{Geometric pro-\boldmath{$\ell$} completions.}
Throughout the following subsection, let $G$ be a strongly complete profinite group, i.e., every subgroup of finite index is open.  
Let us shortly discuss pro-$\ell$-completion in $\mathcal{H}({\hat{\mathcal{S}}}_G)$.

\begin{sect}\label{para: geometric pro l completion}
In \cite{Quick12}, Quick gave an explicit construction of a pro-finite completion in $\hat{\mathcal{S}}_G$.
An analogous construction gives a pro-$\ell$-completion in $\hat{\mathcal{S}}_G$, too (see op.~cit.~Rem.~3.3).
Let us shortly describe this construction:
By op.~cit.~4.3, any profinite $G$-space $\mathcal{X}$ is isomorphic to a profinite $G$-space given by a pro-system $\{ \mathcal{X}_i \}_{i\in I}$ of degreewise finite discrete $G$-spaces $\mathcal{X}_i$.
Then the pro-$\ell$-completion is given as the profinite $G$-space
\begin{equation*}
 \mathcal{X}_\ell^\wedge := \{ \bar{W} \hat{\Gamma}_\ell(\mathcal{X}_i) \}_{i\in I},
\end{equation*}
where $\bar{W}(-)$ is (levelwise) the classifying space and $\hat{\Gamma}_\ell(\mathcal{X}_i)$ is degreewise the pro-$\ell$ completion of the free loop group $\Gamma(\mathcal{X}_i)$ of $\mathcal{X}_i$.
Arguing directly using the (levelwise) homotopy fibre sequence
\begin{equation*}
 \xymatrix{
  \hat{\Gamma}_\ell(\mathcal{X}_i) \ar[r] &
  W \hat{\Gamma}_\ell(\mathcal{X}_i) \ar[r] &
  \bar{W}\hat{\Gamma}_\ell(\mathcal{X}_i)
 },
\end{equation*}
or comparing $\mathcal{X}_\ell^\wedge$ with the fibrant replacement in Morel's pro-$\ell$ model structure (see \cite[Sect.~2.1]{Morel96}), we get that $\pi_1(\mathcal{X}_\ell^\wedge)$ equals the pro-$\ell$ completion $\pi_1^\ell(\mathcal{X})$, $\mathcal{X}_\ell^\wedge$ has pro-$\ell$ homotopy groups and the canonical map $\mathcal{X} \rightarrow \mathcal{X}_\ell^\wedge$ induces an isomorphism in $\mathcal{D}^+(\underline{\rm Mod}_G)$ on cohomology cochains $C^\bullet(-,\Lambda)$ for any finite $\ell$-torsion $G$-module $\Lambda$.
\\
For $\mathcal{X}/BG$ in $\hat{\mathcal{S}}_{/ BG}$, denote by $\mathcal{X}_{(\ell)}^\wedge / BG$ the homotopy type in $\mathcal{H}(\hat{\mathcal{S}}_{/ BG})$ corresponding to the pro-$\ell$ completion $(\mathcal{X} \times_{BG}EG)_\ell^\wedge$ in $\mathcal{H}(\hat{\mathcal{S}}_G)$.
The canonical map $\mathcal{X} \rightarrow \mathcal{X}_{(\ell)}^\wedge$ induces isomorphism in $\mathcal{D}^+(\underline{\rm Ab})$ on cohomology cochains $C^\bullet(-,\Lambda)$ for any finite $\ell$-torsion $G$-module $\Lambda$:
Indeed, $C^\bullet(\mathcal{X}\times_{BG}EG,\Lambda) \rightarrow C^\bullet((\mathcal{X}\times_{BG}EG)_{\ell}^\wedge,\Lambda)$ induces an isomorphism between the respective Hochschild-Serre spectral sequences. 
\end{sect}

If $G$ itself is not a pro-$\ell$ group, $\mathcal{X}_{(\ell)}^\wedge / BG$ corresponds to a ``geometric'' pro-$\ell$ completion in the relative homotopy category $\mathcal{H}(\hat{\mathcal{S}}_{/ BG})$.
Let us discuss the case of $B\pi \rightarrow BG$ for suitable quotients $\pi\twoheadrightarrow G$:

\begin{sect}\label{para: geometric pro l completion of a group}
Let $p\colon \pi\twoheadrightarrow G$ be an epimorphism of profinite groups with kernel $\pi^\prime \unlhd \pi$.
Assume $\pi^\prime$ is an $\ell$-good profinite group.
Let $\Delta_\ell \unlhd_c \pi^\prime$ be the kernel of the pro-$\ell$ completion $\pi^\prime \rightarrow \hat{\pi}_\ell^\prime$.
Note that it is a characteristic subgroup by the universal property of the completion.
Then we define the geometric pro-$\ell$ completion of $\pi\rightarrow G$ as $\hat{\pi}_{(\ell)}:= \pi / \Delta_\ell \rightarrow G$.
By construction, the geometric pro-$\ell$ completion $\pi \twoheadrightarrow \hat{\pi}_{(\ell)}$ sits in the following commutative diagram with exact rows
\begin{equation*}
 \xymatrix{
  \phantom{.}\boldsymbol{1} \ar[r] &
  \pi^\prime \ar[r] \ar@{->>}[d] &
  \pi \ar[r] \ar@{->>}[d] &
  G \ar[r] \ar@{=}[d] &
  \boldsymbol{1}\phantom{.}
 \\
  \phantom{.}\boldsymbol{1} \ar[r] &
  \hat{\pi}_\ell^\prime \ar[r] &
  \hat{\pi}_{(\ell)} \ar[r] &
  G \ar[r] &
  \boldsymbol{1}.
 }
\end{equation*}
It follows that $B\hat{\pi}_{(\ell)}\times_{BG}EG \rightarrow (B\hat{\pi}_{(\ell)}\times_{BG}EG)_\ell^\wedge$ is a weak equivalence in $\hat{\mathcal{S}}_G$ (this is a special case of the pro-$\ell$ analogue of \cite[Thm.\ 3.14]{Quick12}).
Further, the canonical Map $\pi \twoheadrightarrow \hat{\pi}_{(\ell)}$ induces an isomorphism $(B\pi\times_{BG}EG)_\ell^\wedge \rightarrow B\hat{\pi}_{(\ell)}\times_{BG}EG$.
In particular, $(B\pi)_{(\ell)}^\wedge = B (\hat{\pi}_{(\ell)})$.
\end{sect}

For fundamental groups of geometrically connected $K$-varieties we define:

\begin{sect}\label{para: geometric pro l completion of fundamental groups}
We write $\pi_1^\ell(Y)$ (resp.\ $\pi_1^{(\ell)}(Y)$) for the pro-$\ell$ completion (resp.\ geometric pro-$\ell$ completion) in the sense of \ref{para: geometric pro l completion of a group} of the \'{e}tale fundamental group $\pi_1(Y)$ of a geometrically connected $K$-variety $Y$.
Denote by $\pi_1^{(\ell)}(Y/\mathbb{F})$ the geometrically pro-$\ell$ completed fundamental group sequence sequence
\begin{equation*}
 \xymatrix{
  \phantom{.}\boldsymbol{1} \ar[r] &
  \pi_1^\ell(Y\otimes_{K}K^{\rm s}) \ar[r] &
  \pi_1^{(\ell)}(Y) \ar[r] &
  G_{K} \ar[r] &
  \boldsymbol{1}.
 }
\end{equation*}
Say $G_K$ is strongly complete (e.g., $K$ a finite or $p$-adic field).
Then $\pi_1^{(\ell)}(Y)$ is the fundamental group of the geometric pro-$\ell$ completion $\widehat{\rm Et}(Y/K)_{(\ell)}^\wedge$.
\end{sect}

\section{The cycle class of a homotopy fixed point}\label{sect: cycle class of a section}
In this section, we recall and reformulate the definition of the cycle class of a section $s$ of a smooth projective curve $Y/K$ of positive genus.
We extent this definition to homotopy rational and homotopy fixed points of special fibres of regular proper flat models of ($p$-adic) curves. 

\subsection*{Smooth curves.}
Let $Y/K$ be a smooth projective curve over a field $K$ of genus $g\geq 1$ and $s\colon G_K \rightarrow \pi_1(Y)$ a section of the fundamental group sequence $\pi_1(Y/K)$.
Fix an integer $n$ prime to the characteristic of $K$.
Let us first recall and reformulate the definition of the cycle class of $s$:

\begin{sect}\label{para: definition of the cycle class}
In \cite[Thm.~26]{EsnaultWittenberg09}, the cycle class ${\rm cl}_s$ of the section $s$ in ${\rm H}^2(Y,\boldsymbol{\mu}_n)$ is defined as follows: 
${\rm cl}_s$ is the unique class of ${\rm H}^2(Y,\boldsymbol{\mu}_n)$ restricting to the first Chern-class $\hat{c}_1[\Gamma_p]$ in $H^2(Y\times_KY^\prime,\boldsymbol{\mu}_n)$ along the projection $Y\times_KY^\prime \rightarrow Y$.
Here, $\Gamma_p \colon Y^\prime \rightarrow Y\times_KY^\prime$ is the graph of a fine enough neighbourhood $p\colon (Y^\prime,s^\prime) \rightarrow (Y,s)$ of $s$, i.e., a finite \'{e}tale covering $p$ together with a section $s^\prime$ of $\pi_1(Y^\prime/K)$, compatible with the original section $s$.
Note that $\hat{c}_1[\Gamma_p]$ is just the restriction of the Chern-class $\hat{c}_1[\Delta_Y]$ of the diagonal $\Delta_Y \colon Y \hookrightarrow Y\times_KY$ along $Y\times_KY^\prime \rightarrow Y\times_KY$.
\end{sect}

We want to reformulate the definition of ${\rm cl}_s$ in terms of the universal neighbourhood
\begin{equation*}
 \tilde{p}\colon
 \xymatrix{
  (\tilde{Y}_s,\tilde{s}) \ar[r] &
  (Y,s)
 }
\end{equation*}
of $s$, the pro-\'{e}tale covering given by the closed subgroup $s(G_K)\leq \pi_1(Y)$ and the canonical section $\tilde{s}\colon G_K \cong s(G_K) = \pi_1(\tilde{Y}_s)$.
Since $Y$ is a $K(\pi,1)$ space, the \'{e}tale cohomology of $\tilde{Y}_s$ with locally constant coefficients is given as Galois-cohomology of $G_K$.

\begin{sect}\label{para: definition of the cycle class, reformulation I}
Let $f$ (resp.\ $f_{Y^\prime}$) be the structural map $Y\rightarrow S$ (resp.\ $Y^\prime \rightarrow S$) for $S$ the spectrum of $K$.
By the projection formula, $\mathbb{R}(f\times_S f_{Y^\prime})_\ast \boldsymbol{\mu}_n \simeq \mathbb{R}f_\ast \boldsymbol{\mu}_n \otimes^{\mathbb{L}} \mathbb{R}f_{Y^\prime,\ast} \mathbb{Z}/n$.
Taking the colimit over all neighbourhoods $(Y^\prime,s^\prime)$ we get a canonical map
\begin{equation*}
 \tilde{p}^\ast\colon
 \xymatrix{
 \mathbb{R}f_\ast\mathbb{Z}/n \ar[r] &
 \mathbb{R}f_{\tilde{Y}_s,\ast}\mathbb{Z}/n 
 }
 \simeq \mathbb{Z}/n.
\end{equation*}
Tensoring with the identity of $\mathbb{R}f_\ast\boldsymbol{\mu}_n$ and taking Galois-invariants $\mathbb{H}^2(G_K,-)$ we get the pullback map
\begin{equation*}
 {\rm id}\otimes^{\mathbb{L}} \tilde{p}^\ast\colon
 \xymatrix{
  {\rm H}^2(Y\times_KY,\boldsymbol{\mu}_n)\ar[r] &
  {\rm H}^2(Y,\boldsymbol{\mu}_n)
 },
\end{equation*}
which agrees with the colimit of the restriction map ${\rm H}^2(Y\times_KY,\boldsymbol{\mu}_n) \rightarrow {\rm H}^2(Y\times_KY^\prime,\boldsymbol{\mu}_n)$ over all neighbourhoods $(Y^\prime,s^\prime)$ of $s$ by construction.
In particular, we get
\begin{equation*}
 {\rm cl}_s = ({\rm id}\otimes^{\mathbb{L}} \tilde{p}^\ast)(\hat{c}_1[\Delta_Y]).
\end{equation*}
\end{sect}

Since $Y$ is $K(\pi,1)$, the section $s$ induces a homotopy rational point of $Y/K$.
In $\mathcal{H}(\hat{\mathcal{S}}_{G_K})$, $s$ corresponds to a homotopy fixed point $EG_K \rightarrow \widehat{\rm Et}(Y/K)\times_{BG_K}EG_K$.
By abuse of notation, we will denote both the corresponding homotopy rational point and homotopy fixed point by $s$, too.
A key observation is that $\tilde{p}\colon \tilde{Y}_s \rightarrow Y$ is an algebraic model of the homotopy rational point induced by $s$: via $\tilde{s}$, the morphism induced by $\tilde{p}$ on profinite homotopy types is isomorphic to the homotopy rational point given by $s$.  
Using this observation, we want to reformulate the definition of ${\rm cl}_s$ in terms of the induced homotopy rational point or homotopy fixed point of $Y/K$:

\begin{sect}\label{para: definition of the cycle class, reformulation II}
Let $\mathbb{Z}/n \rightarrow \mathcal{I}^\bullet$ be an injective resolution in ${\rm Sh}(Y_{\rm \acute{e}t},\mathbb{Z}/n)$.
Then $\mathbb{R}f_{Y^\prime,\ast}\mathbb{Z}/n$ corresponds to the complex of $G_K$-modules $\Gamma(Y^\prime \otimes_KK^{\rm s},\mathcal{I}^\bullet)$.
By Verdier's hypercover theorem, the latter complex is quasi-isomorphic to the colimit of total complexes of the double complexes $\Gamma(\mathfrak{V}_{\bullet_{\rm I}},\mathcal{I}^{\bullet_{\rm II}})$ over \'{e}tale hypercovers $\mathfrak{V}. \rightarrow Y^\prime\otimes_KK^{\rm s}$ (here $_{\bullet_{\rm I}}$ and $^{\bullet_{\rm II}}$ indicate the directions of the differentials in the double complex). 
Since $\mathbb{R}f_{Y^\prime,\ast}\mathbb{Z}/n$ has cohomology only in degrees in the interval $[0,2]$, we may restrict to truncated hypercovers.
In particular, we may assume that each relevant hypercover has only finitely many Galois-conjugates.
The fibre product of all the Galois-conjugates of one such hypercover descends, i.e., we get
\begin{equation*}
 \mathbb{R}f_{Y^\prime,\ast}\mathbb{Z}/n \simeq
 \colim_{\mathfrak{U}.\rightarrow Y^\prime} {\rm tot}^\bullet(\Gamma(\mathfrak{U}_{\bullet_{\rm I}}\otimes_KK^{\rm s},\mathcal{I}^{\bullet_{\rm II}})) \simeq
 \colim_{\mathfrak{U}.\rightarrow Y^\prime} \Gamma(\mathfrak{U}_\bullet\otimes_KK^{\rm s},\mathbb{Z}/n),
\end{equation*}
in $\mathcal{D}^+(\underline{\rm Mod}_{G_K})$, where $\mathfrak{U}.\rightarrow Y^\prime$ runs through the \'{e}tale hypercovers of $Y^\prime$.
The latter complex is quasi-isomorphic to the cohomology cochains $C^\bullet(\widehat{\rm Et}(Y^\prime/K)\times_{BG_K}EG_K,\mathbb{Z}/n)$ in $\mathcal{D}^+(\underline{\rm Mod}_{G_K})$.
Summing up, we get that $\tilde{p}^\ast\colon \mathbb{R}f_\ast\mathbb{Z}/n \rightarrow \mathbb{R}f_{\tilde{Y}_s,\ast}\mathbb{Z}/n \simeq \mathbb{Z}/n$ is isomorphic to
\begin{equation*}
 s^\ast\colon
 \xymatrix{
  C^\bullet(\widehat{\rm Et}(Y/K)\times_{BG_K}EG_K,\mathbb{Z}/n) \ar[r] &
  C^\bullet(EG_K,\mathbb{Z}/n)
 }
 \simeq \mathbb{Z}/n.
\end{equation*}
Tensoring with the identity of $C^\bullet(\widehat{\rm Et}(Y/K)\times_{BG_K}EG_K,\boldsymbol{\mu}_n)$ and taking $\mathbb{H}^2(G_K,-)$ we get ${\rm id}\otimes^{\mathbb{L}} \tilde{p}^\ast$ as the pullback map
\begin{equation*}
 {\rm id}\otimes^{\mathbb{L}} s^\ast\colon
 \xymatrix{
  {\rm H}^2(Y\times_KY,\boldsymbol{\mu}_n)\ar[r] &
  {\rm H}^2(Y,\boldsymbol{\mu}_n)
 },
\end{equation*}
i.e., we get ${\rm cl}_s = ({\rm id}\otimes^{\mathbb{L}} s^\ast)(\hat{c}_1[\Delta_Y])$.
\end{sect}

\begin{sect}\label{para: cycle class of a ho rat point}
Using the characterization of ${\rm cl}_s$ in \ref{para: definition of the cycle class, reformulation II} as a definition, we get a cycle class ${\rm cl}_s$ of any homotopy rational point $s\colon BG_K \rightarrow \widehat{\rm Et}(Y/K)$ or homotopy fixed point of an arbitrary smooth curve $Y/K$:
We set ${\rm cl}_s = ({\rm id}\otimes^{\mathbb{L}} s^\ast)(\hat{c}_1[\Delta_Y])$ for ${\rm id}\otimes^{\mathbb{L}} s^\ast$ the map induced on $\mathbb{H}^2(G_K,-)$ by the derived tensor product in $\mathcal{D}^+(\underline{\rm Mod}_{G_K})$ of 
\begin{equation*}
 s^\ast\colon
 \xymatrix{
  C^\bullet(\widehat{\rm Et}(Y/K)\times_{BG_K}EG_K,\mathbb{Z}/n) \ar[r] &
  \mathbb{Z}/n
 }
\end{equation*}
and the identity of $C^\bullet(\widehat{\rm Et}(Y/K)\times_{BG_K}EG_K,\boldsymbol{\mu}_n)$.
If we work with the cycle class of the diagonal of $Z$ in ${\rm H}^{2d}(Z\times_KZ,\mathbb{Z}/n(d))$, we get a cycle class for homotopy rational points of $d$-dimensional smooth $K$-varieties $Z/K$, too. 
\end{sect}

\begin{sect}\label{para: cycle class of a ho rat point up to L completion}
Let $L$ be a set of primes containing all prime divisors of $n$.
Then the definition of ${\rm cl}_s$ in \ref{para: cycle class of a ho rat point} still makes sense if we start only with a homotopy fixed point $s$ of the pro-$L$-completion $(\widehat{\rm Et}(Y/K)\times_{BG_K}EG_K)_L^\wedge$ in $\mathcal{H}(\hat{\mathcal{S}}_{G_K})$ (see \cite{Quick12}, in particularly Rem.~3.3, see also \ref{para: geometric pro l completion} for $L=\{ \ell \}$).
\end{sect}

\subsection*{Special fibres of regular models of (\boldmath{$p$}-adic) curves.}
The diagonal $\Delta_Y$ is Cartier if and only if $Y/K$ is smooth.
Still, for a homotopy rational or homotopy fixed point of the (possibly non smooth) special fibre of a proper flat regular model of a curve we can define a cycle class.
Fix a $p$-adic field $k$ with ring of integers $\mathfrak{o}$ and residue field $\mathbb{F}$ (or more generally: $\mathfrak{o}$ a henselian discrete valuation ring), a smooth curve $X/k$ and a proper flat regular model $X\hookrightarrow \mathfrak{X}$ over $\mathfrak{o}$ with reduced special fibre $Y = (\mathfrak{X}\otimes_{\mathfrak{o}}\mathbb{F})_{\rm red}$.
Let ${\bar{s}}$ be a homotopy rational or homotopy fixed point of the special fibre $Y/\mathbb{F}$.
Fix an integer $n$ prime to the residue characteristic.
To define the cycle class ${\rm cl}_{\bar{s}}$ in ${\rm H}^2(Y,\boldsymbol{\mu}_n) = {\rm H}^2(\mathfrak{X},\boldsymbol{\mu}_n)$, let us first find a suitable replacement of the pullback map ${\rm id}\otimes^{\mathbb{L}} {\bar{s}}^\ast$ in \ref{para: cycle class of a ho rat point}:

\begin{sect}\label{para: pullback map}
Let $\mathfrak{X}^\bullet/\mathfrak{o}$ be the punctured model $\mathfrak{X}\setminus {\rm Sing}(Y)$.
Let $f\colon \mathfrak{X}\rightarrow S$ and $f^\bullet\colon \mathfrak{X}^\bullet\rightarrow S$ be the structural maps for $S$ the spectrum of $\mathfrak{o}$.
Since $\mathfrak{o}$ is a henselian local ring, we have ${\rm H}^2(\mathfrak{X}^\bullet,\boldsymbol{\mu}_n) = \mathbb{H}^2(G_{\mathbb{F}},\sigma^\ast\mathbb{R}f_\ast^\bullet\boldsymbol{\mu}_n)$.
Similarly, using the projection formula ($f^\bullet$ is no longer proper, but we can use e.g.~\cite[Prop.~6.5.5]{Fu11} instead), we have
\begin{equation*}
{\rm H}^2(\mathfrak{X}^\bullet\times_{\mathfrak{o}}\mathfrak{X},\boldsymbol{\mu}_n) = \mathbb{H}^2(G_{\mathbb{F}},\sigma^\ast\mathbb{R}f_\ast^\bullet\boldsymbol{\mu}_n\otimes^{\mathbb{L}}\sigma^\ast\mathbb{R}f_\ast\mathbb{Z}/n). 
\end{equation*}
By proper base change and topological invariance, $\sigma^\ast\mathbb{R}f_\ast\mathbb{Z}/n$ is quasi-isomorphic to $\mathbb{R}f_{\sigma,\ast}\mathbb{Z}/n$ for $f_\sigma\colon Y \rightarrow {\rm Spec}(\mathbb{F})$ the structural map.
As explained in \ref{para: definition of the cycle class, reformulation II}, we thus get a pullpack map ${\bar{s}}^\ast\colon \sigma^\ast\mathbb{R}f_\ast\mathbb{Z}/n \rightarrow \mathbb{Z}/n$.
Tensoring with the identity of $\sigma^\ast\mathbb{R}f_\ast^\bullet\boldsymbol{\mu}_n$ and taking $\mathbb{H}^2(G_{\mathbb{F}},-)$ we get a pullback map
\begin{equation*}
{\rm id}\otimes^{\mathbb{L}} {\bar{s}}^\ast\colon
 \xymatrix{
  {\rm H}^2(\mathfrak{X}^\bullet\times_{\mathfrak{o}}\mathfrak{X},\boldsymbol{\mu}_n)\ar[r] &
  {\rm H}^2(\mathfrak{X}^\bullet,\boldsymbol{\mu}_n)
 }.
\end{equation*}
\end{sect}

\begin{sect}\label{para: pullback map equivariant}
Moreover, suppose $k/k^G$ is an unramified Galois extension with group $G$, the model $X \hookrightarrow \mathfrak{X}$ is a base-extensions along $\mathfrak{o}/\mathfrak{o}^G$ of a model $X_G\hookrightarrow \mathfrak{X}_G$ and ${\bar{s}}$ is $G$-equivariant (in the week sense that it is contained in $[EG_{\mathbb{F}},\widehat{\rm Et}(Y/\mathbb{F})\times_{BG_{\mathbb{F}}}EG_{\mathbb{F}}]_{\hat{\mathcal{S}}_{G_{\mathbb{F}}}}^G$).
Then the pullback map constructed in \ref{para: pullback map} is $G$-equivariant.
In particular, it induces a map
\begin{equation*}
 \xymatrix{
  {\rm H}^2(\mathfrak{X}^\bullet\times_{\mathfrak{o}}\mathfrak{X},\boldsymbol{\mu}_n)^G\ar[r] &
  {\rm H}^2(\mathfrak{X}^\bullet,\boldsymbol{\mu}_n)^G
 }.
\end{equation*}
\end{sect}

By Gabber's Absolute Purity Theorem (see \cite{Fujiwara02}), ${\rm H}^2(\mathfrak{X},\boldsymbol{\mu}_n) \cong {\rm H}^2(\mathfrak{X}^\bullet,\boldsymbol{\mu}_n)$, allowing us to define ${\rm cl}_{\bar{s}}$ as a pullback of the cycle class of the restricted diagonal:

\begin{sect}\label{para: cycle class of a ho rat point of the special fibre}
The diagonal $\Delta_{\mathfrak{X}}\colon \mathfrak{X} \hookrightarrow \mathfrak{X}\times_{\mathfrak{o}}\mathfrak{X}$ fails to be Cartier only at points $(y,y)$ for $y$ a singular point of the special fibre $Y$.
Thus, $\Delta_{\mathfrak{X}}\vert_{\mathfrak{X}^\bullet\times_{\mathfrak{o}}\mathfrak{X}}$ is Cartier.
Using the pullback map in \ref{para: pullback map}, we get a well defined class $({\rm id}\otimes^{\mathbb{L}} {\bar{s}}^\ast)(\hat{c}_1[\Delta_{\mathfrak{X}}\vert_{\mathfrak{X}^\bullet\times_{\mathfrak{o}}\mathfrak{X}}])$ in ${\rm H}^2(\mathfrak{X}^\bullet,\boldsymbol{\mu}_n)$.
The singular locus ${\rm Sing}(Y)$ has codimension $2$ in the regular scheme $\mathfrak{X}$.
It follows from the Absolute Purity Theorem that the canonical map ${\rm H}^2(\mathfrak{X},\boldsymbol{\mu}_n) \rightarrow {\rm H}^2(\mathfrak{X}^\bullet,\boldsymbol{\mu}_n)$ is an isomorphism.
\end{sect}

Thus we define:

\begin{defin}\label{def: cycle class of ho rat points of special fibre}
Let $Y = (\mathfrak{X}\otimes_{\mathfrak{o}}\mathbb{F})_{\rm red}$ be the reduced special fibre of a regular proper flat model $\mathfrak{X}/\mathfrak{o}$ of a smooth projective $p$-adic curve.
Let $\bar{s}$ be a homotopy fixed or homotopy rational point of $Y/\mathbb{F}$.
For $n$ invertible in $\mathfrak{o}$, we define its cycle class ${\rm cl}_{\bar{s}}$ as the class in ${\rm H}^2(\mathfrak{X},\boldsymbol{\mu}_n)$ corresponding to $({\rm id}\otimes^{\mathbb{L}} {\bar{s}}^\ast)(\hat{c}_1[\Delta_{\mathfrak{X}}\vert_{\mathfrak{X}^\bullet\times_{\mathfrak{o}}\mathfrak{X}}])$.
\end{defin}

\begin{sect}\label{para: cycle class of a ho rat point of the special fibre first properties}
By definition, the cycle classes ${\rm cl}_{\bar{s}}$ are compatible for various $n$.
Further, in the situation of \ref{para: pullback map equivariant}, ${\rm cl}_{\bar{s}}$ is $G$-invariant, i.e., contained in ${\rm H}^2(\mathfrak{X},\boldsymbol{\mu}_n)^G$:
Indeed, $\Delta_{\mathfrak{X}}\vert_{\mathfrak{X}^\bullet\times_{\mathfrak{o}}\mathfrak{X}}$ is $G$-invariant as a base-extension along $\mathfrak{o}^\prime/\mathfrak{o}$ and the pullback map is $G$-equivariant by \ref{para: pullback map equivariant}.
\end{sect}

\begin{sect}\label{para: cycle class of a ho rat point of the special fibre up to L completion}
Let $L$ be a set of primes containing all the prime divisors of $n$.
As in \ref{para: cycle class of a ho rat point up to L completion}, a homotopy fixed point ${\bar{s}}$ of the pro-$L$-completion $(\widehat{\rm Et}(Y/\mathbb{F})\times_{BG_{\mathbb{F}}}EG_{\mathbb{F}})_L^\wedge$ in $\mathcal{H}(\hat{\mathcal{S}}_{G_{\mathbb{F}}})$ suffices to define the cycle class ${\rm cl}_{\bar{s}}$ in ${\rm H}^2(\mathfrak{X},\boldsymbol{\mu}_n)$.
\end{sect}

\section{A canonical lift of the cycle class to the model}

Fix a $p$-adic field $k$ with ring of integers $\mathfrak{o}$ and residue field $\mathbb{F}$, a smooth geometrically connected projective curve $X/k$ of positive genus and a proper flat regular model $X\hookrightarrow \mathfrak{X}$ over $\mathfrak{o}$ with reduced special fibre $Y = (\mathfrak{X}\otimes_{\mathfrak{o}}\mathbb{F})_{\rm red}$.
Let $\eta$ (resp.\ $\sigma$) be the generic point (resp.\ closed point) of $S={\rm Spec}(\mathfrak{o})$ and denote the respective structural morphisms of $X/k$, $\mathfrak{X}/\mathfrak{o}$ and $Y/\mathbb{F}$ by $f_\eta$, $f$ and $f_\sigma$. 
Let $s$ be a section of the fundamental group sequence $\pi_1(X/k)$.
In abuse of notation, denote by $s$ the induced homotopy rational and homotopy fixed point of $X/k$, as well.
Our aim is to give constructions of canonical lifts of the cycle classes ${\rm cl}_s$ in ${\rm H}^2(X,\boldsymbol{\mu}_{\ell^n})$ to ${\rm H}^2(\mathfrak{X},\boldsymbol{\mu}_{\ell^n})$ for $\ell \neq p$, proving Thm.~\ref{mainthm A}.

\subsection*{A general recipe.}
Let us first develop a general recipe for the construction of a canonical lift of ${\rm cl}_s$ to the cohomology of the model.

\begin{sect}\label{para: quasi-specialized ho rat points}
Let $\bar{s}_\ell$ be a homotopy fixed point of the pro-$\ell$-completion $(\widehat{\rm Et}(Y/\mathbb{F})\times_{BG_{\mathbb{F}}}EG_{\mathbb{F}})_\ell^\wedge$.
The geometric fundamental group $\pi_1(Y\otimes_{\mathbb{F}}\mathbb{F}^{\rm s})$ is $\ell$-good (see Lem.~\ref{lem: homotopy type of a curve}).
Using \ref{para: geometric pro l completion of a group}, $\bar{s}_\ell$ induces a homotopy fixed point of $B\pi_1^{(\ell)}(Y) \times_{BG_{\mathbb{F}}}EG_{\mathbb{F}}$, i.e.\ a section of the geometrically pro-$\ell$ completed fundamental group sequence $\pi_1^{(\ell)}(Y/\mathbb{F})$.
In abuse of notation, let us denote this section by $\bar{s}_\ell$, as well.
We call the homotopy fixed $\bar{s}_\ell$ a {\bf quasi-specialization} of $s$, 
if the induced section is a {\bf specialization} of the section $s$ to $\pi_1^{\ell}(Y/\mathbb{F})$ (cf.~\cite[Chap.~8]{Stix12}), i.e., if the canonical square
\begin{equation*}
 \xymatrix{
  G_k \ar[r]^-s \ar[d] &
  \pi_1(X) \ar[d]^-{\rm sp}
 \\
  G_{\mathbb{F}} \ar[r]^-{\bar{s}_\ell} &
  \pi_1^{(\ell)}(Y)
 }
\end{equation*}
commutes.
This is equivalent to the triviality of the map ${\rm ram}_s^{(\ell)}\colon I_k \rightarrow \pi_1^{(\ell)}(Y)$ induced by the ramification map
\begin{equation*}
 {\rm ram}_s\colon
 \xymatrix{
  I_k \ar@{^(->}[r] \ar@/_1pc/[rrr] &
  G_k \ar[r]^-s &
  \pi_1(X) \ar@{->>}[r]^-{\rm sp} &
  \pi_1(Y)
 }.
\end{equation*}
\end{sect}

By properness, $\widehat{\rm Et}(Y)\rightarrow \widehat{\rm Et}(\mathfrak{X})$ is a weak equivalence (e.g.~\cite[Thm.~1.2, Prop.~2.1 and Prop.~2.2]{SGA4.5}).
Via the canonical epimorphism of Galois groups $\kappa\colon G_k\twoheadrightarrow G_{\mathbb{F}}$, we get the specialization map
\begin{equation*}
 {\rm sp}\colon
 \xymatrix{
  \widehat{\rm Et}(X/k)\ar[r] &
  \widehat{\rm Et}(Y/\mathbb{F})
 }
 {\rm ~resp.~}
 \xymatrix{
  \widehat{\rm Et}(X/k)\times_{BG_{\mathbb{F}}}EG_{\mathbb{F}} \ar[r] &
  \widehat{\rm Et}(Y/\mathbb{F})\times_{BG_{\mathbb{F}}}EG_{\mathbb{F}}
 }
\end{equation*}
in the relative homotopy category $\mathcal{H}(\hat{\mathcal{S}}_{/ BG_{\mathbb{F}}})$ resp.\ in the equivariant homotopy category $\mathcal{H}(\hat{\mathcal{S}}_{G_{\mathbb{F}}})$.
The induced maps on fundamental groups and cohomology are just the usual respective specialization maps.

\begin{sect}\label{para: admissible pull back map}
Let $\bar{s}_\ell$ be a homotopy fixed point of the pro-$\ell$-completion $(\widehat{\rm Et}(Y/\mathbb{F})\times_{BG_{\mathbb{F}}}EG_{\mathbb{F}})_\ell^\wedge$.
Assume it is a quasi-specialization of $s$.
For the moment, let us call the induced pullback map on cohomology cochains $\bar{s}_\ell^\ast\colon C^\bullet(\widehat{\rm Et}(Y/\mathbb{F})\times_{BG_{\mathbb{F}}}EG_{\mathbb{F}},\mathbb{Z}/\ell^n) \rightarrow \mathbb{Z}/\ell^n$ {\bf admissible}, if
\begin{equation}\label{eq: specialization in cohomological settings}
 \xymatrix{
  C^\bullet(\widehat{\rm Et}(Y/\mathbb{F})\times_{BG_{\mathbb{F}}}EG_{\mathbb{F}},\mathbb{Z}/\ell^n) \ar[r]^-{\bar{s}_\ell^\ast} \ar[d]^-{{\rm sp}^\ast} &
  C^\bullet(EG_{\mathbb{F}},\mathbb{Z}/\ell^n)\simeq \mathbb{Z}/\ell^n \ar[d]
 \\
  C^\bullet(\widehat{\rm Et}(X/k)\times_{BG_{\mathbb{F}}}EG_{\mathbb{F}},\mathbb{Z}/\ell^n) \ar[r]^-{s^\ast} &
  C^\bullet(BG_k\times_{BG_{\mathbb{F}}}EG_{\mathbb{F}},\mathbb{Z}/\ell^n)
 }
\end{equation}
is a commutative square in $\mathcal{D}^+(\underline{\rm Mod}_{G_{\mathbb{F}}})$.
\end{sect}

\begin{rem}\label{rem: admissible pull back map subtlety}
A priori, it is not clear that an arbitrary quasi-specialization $\bar{s}_\ell$ of $s$ has an admissible pullback map:
We only know that $\bar{s}_\ell$ and $s$ are compatible on fundamental groups so we only get the commutative diagram
\begin{equation*}
 \xymatrix{
  C^\bullet(B\Pi(\widehat{\rm Et}(Y/\mathbb{F})\times_{BG_{\mathbb{F}}}EG_{\mathbb{F}}),\mathbb{Z}/\ell^n) \ar[r]^-{\bar{s}_\ell^\ast} \ar[d]^-{{\rm sp}^\ast} &
  C^\bullet(EG_{\mathbb{F}},\mathbb{Z}/\ell^n)\simeq \mathbb{Z}/\ell^n \phantom{.} \ar[d]
 \\
  C^\bullet(B\Pi(\widehat{\rm Et}(X/k)\times_{BG_{\mathbb{F}}}EG_{\mathbb{F}}),\mathbb{Z}/\ell^n) \ar[r]^-{s^\ast} &
  C^\bullet(BG_k\times_{BG_{\mathbb{F}}}EG_{\mathbb{F}},\mathbb{Z}/\ell^n).
 }
\end{equation*}
Fortunately, in Thm.~\ref{thm: specialized homotopy rational point pro l}, below, we will prove that a posteriori every quasi-specialization $\bar{s}_\ell$ has an admissible pullback map.
\end{rem}

\begin{lem}\label{lem: canonical lift of the pullback map}
Let $\bar{s}_\ell$ be a homotopy fixed point of $(\widehat{\rm Et}(Y/\mathbb{F})\times_{BG_{\mathbb{F}}}EG_{\mathbb{F}})_\ell^\wedge$.
Suppose $\bar{s}_\ell$ is a quasi-specialization of $s$ and its induced pullback map $\bar{s}_\ell^\ast$ is admissible.
Then the pullback map ${\rm id}\otimes^{\mathbb{L}} \bar{s}_\ell^\ast$ constructed in \ref{para: pullback map} is compatible with ${\rm id}\otimes^{\mathbb{L}} s^\ast$ (as constructed in \ref{para: definition of the cycle class, reformulation II}), i.e., we have a commutative square
\begin{equation*}
 \xymatrix{
  {\rm H}^2(\mathfrak{X}^\bullet\times_{\mathfrak{o}}\mathfrak{X},\boldsymbol{\mu}_{\ell^n}) \ar[r]^-{{\rm id}\otimes^{\mathbb{L}} \bar{s}_\ell^\ast} \ar[d]^-{\eta^\ast} &
  {\rm H}^2(\mathfrak{X}^\bullet,\boldsymbol{\mu}_{\ell^n}) \phantom{.} \ar[d]^-{\eta^\ast}
 \\
  {\rm H}^2(X\times_kX,\boldsymbol{\mu}_{\ell^n}) \ar[r]^-{{\rm id}\otimes^{\mathbb{L}} s^\ast} &
  {\rm H}^2(X,\boldsymbol{\mu}_{\ell^n}).
 }
\end{equation*}
\end{lem}

\begin{proof}
Arguing similarly to \ref{para: definition of the cycle class, reformulation II}, we get an isomorphism in the derived category $\mathcal{D}^+(\underline{\rm Mod}_{G_{\mathbb{F}}})$ between $\sigma^\ast \mathbb{R}\eta_\ast \mathbb{R}f_{\eta,\ast}\mathbb{Z}/\ell^n$ and $C^\bullet(\widehat{\rm Et}(X/k)\times_{BG_{\mathbb{F}}}EG_{\mathbb{F}},\mathbb{Z}/\ell^n)$, compatible with the isomorphism between $\sigma^\ast \mathbb{R}f_\ast\mathbb{Z}/\ell^n$ and $C^\bullet(\widehat{\rm Et}(Y/\mathbb{F})\times_{BG_{\mathbb{F}}}EG_{\mathbb{F}},\mathbb{Z}/\ell^n)$ via the specializ{\-}ation- and the base change map $\mathbb{R}f_\ast \mathbb{Z}/\ell^n \rightarrow \mathbb{R}\eta_\ast \mathbb{R}f_{\eta,\ast} \mathbb{Z}/\ell^n$. 
Under these identifications, (\ref{eq: specialization in cohomological settings}) translates to the commutative diagram
\begin{equation*}
 \xymatrix{
  \sigma^\ast \mathbb{R}f_\ast\mathbb{Z}/\ell^n \ar[r]^-{\bar{s}_\ell^\ast} \ar[d] & 
  \mathbb{Z}/\ell^n \phantom{.} \ar[d]
 \\
  \sigma^\ast \mathbb{R}\eta_\ast \mathbb{R}f_{\eta,\ast}\mathbb{Z}/\ell^n \ar[r]^-{s^\ast} &
  \sigma^\ast \mathbb{R}\eta_\ast \mathbb{Z}/\ell^n .
 }
\end{equation*}
Tensoring with the morphism $\sigma^\ast\mathbb{R}f_\ast^\bullet\boldsymbol{\mu}_{\ell^n} \rightarrow \sigma^\ast \mathbb{R}\eta_\ast \mathbb{R}f_{\eta,\ast}\boldsymbol{\mu}_{\ell^n}$ (note that $f_\eta^\bullet = f_\eta$) and taking $\mathbb{H}^2(G_{\mathbb{F}},-)$, we can extend this to the commutative diagram
\begin{equation*}
\resizebox{.9\hsize}{!}{
 \xymatrix{
  \mathbb{R}\Gamma(\mathfrak{X}^\bullet\times_{\mathfrak{o}}\mathfrak{X},\boldsymbol{\mu}_{\ell^n}) \ar[r]^-{{\rm id}\otimes^{\mathbb{L}} \bar{s}_\ell^\ast} \ar[d] &
  \mathbb{R}\Gamma(\mathfrak{X}^\bullet,\boldsymbol{\mu}_{\ell^n}) \phantom{.} \ar[d]
 \\
  \mathbb{R}\Gamma(G_{\mathbb{F}},\sigma^\ast \mathbb{R}\eta_\ast \mathbb{R}f_{\eta,\ast}\boldsymbol{\mu}_{\ell^n} \otimes^{\mathbb{L}} \sigma^\ast \mathbb{R}\eta_\ast \mathbb{R}f_{\eta,\ast} \mathbb{Z}/\ell^n) \ar[r] \ar[d] &
  \mathbb{R}\Gamma(G_{\mathbb{F}},\sigma^\ast \mathbb{R}\eta_\ast \mathbb{R}f_{\eta,\ast}\boldsymbol{\mu}_{\ell^n} \otimes^{\mathbb{L}} \sigma^\ast \mathbb{R}\eta_\ast \mathbb{Z}/\ell^n) \phantom{.} \ar[d]
 \\
  \mathbb{R}\Gamma(G_{k},\eta^\ast \mathbb{R}\eta_\ast \mathbb{R}f_{\eta,\ast}\boldsymbol{\mu}_{\ell^n} \otimes^{\mathbb{L}} \eta^\ast \mathbb{R}\eta_\ast \mathbb{R}f_{\eta,\ast} \mathbb{Z}/\ell^n) \ar[r] \ar[d] &
  \mathbb{R}\Gamma(G_{k},\eta^\ast \mathbb{R}\eta_\ast \mathbb{R}f_{\eta,\ast}\boldsymbol{\mu}_{\ell^n} \otimes^{\mathbb{L}} \eta^\ast \mathbb{R}\eta_\ast \mathbb{Z}/\ell^n) \phantom{.}  \ar[d]
 \\
  \mathbb{R}\Gamma(X\times_kX,\boldsymbol{\mu}_{\ell^n}) \ar[r]^-{{\rm id}\otimes^{\mathbb{L}} s^\ast} &
  \mathbb{R}\Gamma(X,\boldsymbol{\mu}_{\ell^n}).
 }}
\end{equation*}
Here the middle square is induced by the canonical epimorphism $\kappa\colon G_k\twoheadrightarrow G_{\mathbb{F}}$. Use that a $\mathbb{Z}/\ell^n[G]$-module is flat if and only if the underlying $\mathbb{Z}/\ell^n$-module is flat and $\kappa^\ast \circ \sigma^\ast$ is just $\eta^\ast$.
The lower square is the canonical square induced by the tensor product of the canonical maps $\eta^\ast\mathbb{R}\eta_\ast \rightarrow {\rm id}$.
Unravelling the definitions, we find that the vertical compositions are indeed the canonical maps induced by $\eta$.
\end{proof}

Plugging in the Chern class $\hat{c}_1[\Delta_{\mathfrak{X}}\vert_{\mathfrak{X}^\bullet\times_{\mathfrak{o}}\mathfrak{X}}]$ of the restricted diagonal, we get a lift of ${\rm cl}_s$ to ${\rm H}^2(\mathfrak{X},\boldsymbol{\mu}_{\ell^n})$:

\begin{cor}\label{cor: lift of the cycle class}
Let $\bar{s}_\ell$ be a homotopy fixed point of $(\widehat{\rm Et}(Y/\mathbb{F})\times_{BG_{\mathbb{F}}}EG_{\mathbb{F}})_\ell^\wedge$.
Suppose it is a quasi-specialization of $s$ (cf.~\ref{para: quasi-specialized ho rat points}) and its induced pullback map $\bar{s}_\ell^\ast$ is admissible (cf.~\ref{para: admissible pull back map}).
Then the cycle class ${\rm cl}_{\bar{s}_\ell} \in {\rm H}^2(\mathfrak{X},\boldsymbol{\mu}_{\ell^n})$ of $\bar{s}_\ell$ restricts to ${\rm cl}_{s} \in {\rm H}^2(X,\boldsymbol{\mu}_{\ell^n})$.
\end{cor}

\subsection*{Quasi-specialized homotopy fixed points.}
By Cor.\ \ref{cor: lift of the cycle class}, we get a canonical lift of the cycle class ${\rm cl}_s$ to ${\rm H}^2(\mathfrak{X},\boldsymbol{\mu}_{\ell^n})$, if we can find a canonical quasi-specialization with admissible pullback map $\bar{s}_\ell$ of $s$ to $(\widehat{\rm Et}(Y/\mathbb{F})\times_{BG_{\mathbb{F}}}EG_{\mathbb{F}})_\ell^\wedge$.
At least a specialized section $\bar{s}_\ell$ of $\pi_1^{(\ell)}(Y/\mathbb{F})$ always does exist:

\begin{lem}\label{lem: unramified on geometric l completion}
Any section $s$ of $\pi_1(X/k)$ specializes to a (necessarily unique) section $\bar{s}_\ell$ of $\pi_1^{(\ell)}(Y/\mathbb{F})$.
\end{lem}

\begin{proof}
This is a direct consequence of \cite[Prop.~91]{Stix12}:
The ramification map (cf.~\ref{para: quasi-specialized ho rat points}) induces the geometrically pro-$\ell$ completed ramification map ${\rm ram}_s^{(\ell)}\colon I_k \rightarrow \pi_1^{(\ell)}(Y)$.
By construction, ${\rm ram}_s^{(\ell)}$ factors over $\pi_1^{\ell}(Y\otimes_{\mathbb{F}}\mathbb{F}^{\rm s})$, i.e., ${\rm ram}_s^{(\ell)}$ is trivial by loc.~cit.. 
\end{proof}

Together with Lem.~\ref{lem: sections vs ho rat points} this implies the existence of a unique quasi-specialized homotopy fixed point:

\begin{cor}\label{cor: unramified on geometric l completion and ho rat pts}
For any (possibly ramified) section $s$ of $\pi_1(X/k)$ there is a unique quasi-specialized homotopy fixed point $\bar{s}_\ell$ of the pro-$\ell$ completion $(\widehat{\rm Et}(Y/\mathbb{F})\times_{BG_{\mathbb{F}}}EG_{\mathbb{F}})_{\ell}^\wedge$ in $\mathcal{H}(\hat{\mathcal{S}}_{G_{\mathbb{F}}})$.
\\
Moreover, suppose $k/k^G$ is an unramified Galois extension with group $G$, the model $X \hookrightarrow \mathfrak{X}$ is a base-extension along $\mathfrak{o}/\mathfrak{o}^G$ of a model $X_G\hookrightarrow \mathfrak{X}_G$ and $s$ is the restriction of a section $s_G$ of $\pi_1(X_G/k^G)$.
Then $\bar{s}_\ell$ is contained in $[EG_{\mathbb{F}},\widehat{\rm Et}(Y/\mathbb{F})\times_{BG_{\mathbb{F}}}EG_{\mathbb{F}}]_{\hat{\mathcal{S}}_{G_{\mathbb{F}}}}^G$.
\end{cor}

\begin{proof}
The Galois group $G_{\mathbb{F}}$ is $\hat{\mathbb{Z}}$, so the first claim follows using Lem.~\ref{lem: sections vs ho rat points}.
For the $G$-equivariance claim, use Lem.~\ref{lem: sections vs ho rat points equivariant}.
\end{proof}

In case the canonical map $(\widehat{\rm Et}(Y/\mathbb{F})\times_{BG_{\mathbb{F}}}EG_{\mathbb{F}})_\ell^\wedge \rightarrow B\Pi((\widehat{\rm Et}(Y/\mathbb{F})\times_{BG_{\mathbb{F}}}EG_{\mathbb{F}})_\ell^\wedge)$ admits a section in $\mathcal{H}(\hat{\mathcal{S}}_{G_{\mathbb{F}}})$, we get more explicitly:

\begin{cor}\label{cor: splitting of the fundamental groupoid}
Suppose the canonical map in $\mathcal{H}(\hat{\mathcal{S}}_{G_{\mathbb{F}}})$ to the fundamental groupoid admits a section $t\colon B\Pi((\widehat{\rm Et}(Y/\mathbb{F})\times_{BG_{\mathbb{F}}}EG_{\mathbb{F}})_\ell^\wedge) \rightarrow (\widehat{\rm Et}(Y/\mathbb{F})\times_{BG_{\mathbb{F}}}EG_{\mathbb{F}})_\ell^\wedge$.
Let $s$ be a section of $\pi_1(X/k)$.
Then the quasi-specialized homotopy fixed point $\bar{s}_\ell$ of the pro-$\ell$ completion $(\widehat{\rm Et}(Y/\mathbb{F})\times_{BG_{\mathbb{F}}}EG_{\mathbb{F}})_{\ell}^\wedge$ given by Cor.~\ref{cor: unramified on geometric l completion and ho rat pts} is given as the composition of the specialized section of Lem.~\ref{lem: unramified on geometric l completion} (seen as a homotopy fixed point of the fundamental groupoid) with $t$.
\end{cor}

\begin{proof}
See the proof of Lem.~\ref{lem: sections vs ho rat points}.
\end{proof}

\begin{sect}\label{para: nice section}
Under a suitable rationality condition on $Y$, there is a particularly nice section of the canonical map $(\widehat{\rm Et}(Y/\mathbb{F})\times_{BG_{\mathbb{F}}}EG_{\mathbb{F}})_\ell^\wedge \rightarrow B\Pi((\widehat{\rm Et}(Y/\mathbb{F})\times_{BG_{\mathbb{F}}}EG_{\mathbb{F}})_\ell^\wedge)$ in $\mathcal{H}(\hat{\mathcal{S}}_{G_{\mathbb{F}}})$:
As in the proof of Lem.~\ref{lem: homotopy type of a curve}, let $\tilde{Y} = \coprod_i\tilde{Y}_i \rightarrow Y$ be the normalization with connected components $\tilde{Y}_i$.
Let $\mathcal{R}$ be the set of indices $i$ with $\tilde{Y}_i$ a rational component of the normalization $\tilde{Y} \rightarrow Y$.
Since finite fields have trivial Brauer groups, there is a finite extension $\mathbb{F}_i/\mathbb{F}$ such that $\tilde{Y}_i \cong \mathbb{P}_{\mathbb{F}_i}^1$ for $i\in \mathcal{R}$.
Assume each rational component $\tilde{Y}_i$ admits an $\mathbb{F}_i$-rational point $\infty_i \in \tilde{Y}_i(\mathbb{F}_i)$ lying over a smooth point of $Y$.
Then $\tilde{Y}_i \setminus \{ \infty_i \} \cong \mathbb{A}_{\mathbb{F}_i}^1$ and $(\tilde{Y}_i \setminus \{ \infty_i \}) \otimes_{\mathbb{F}}\mathbb{F}^{\rm s}$ is isomorphic to a disjoint union of affine lines.
In particular, each connected component of $(\tilde{Y}_i \setminus \{ \infty_i \}) \otimes_{\mathbb{F}}\mathbb{F}^{\rm s}$ has contractible pro-$\ell$ completion.
Set 
\begin{equation*}
 j\colon Y^\circ :=
 \xymatrix{
  Y \setminus \{ \infty_i ~\vert~ i\in \mathcal{R} \} \ar@{^(->}[r] &
  Y
 }.
\end{equation*}
By Lem.~\ref{lem: homotopy type of a curve}, $Y^\circ$ is a $K(\pi,1)$ with an $\ell$-good geometric fundamental group.
Further, $j$ induces an isomorphism between the geometrically completed fundamental groups $\pi_1^{(\ell)}(Y^\circ)$ and $\pi_1^{(\ell)}(Y)$ (use \cite[Cor.~5.3]{Stix06}).
Since $\pi_1(Y^\circ \otimes_{\mathbb{F}}\mathbb{F}^{\rm s})$ is $\ell$-good, $j$ induces the desired splitting
\begin{equation*}
 t_j\colon
 \xymatrix{
  (\widehat{\rm Et}(Y^\circ)\times_{BG_{\mathbb{F}}}EG_{\mathbb{F}})_\ell^\wedge \ar[r] &
  (\widehat{\rm Et}(Y)\times_{BG_{\mathbb{F}}}EG_{\mathbb{F}})_\ell^\wedge
 }
\end{equation*}
of the canonical map $(\widehat{\rm Et}(Y)\times_{BG_{\mathbb{F}}}EG_{\mathbb{F}})_\ell^\wedge \rightarrow B\Pi((\widehat{\rm Et}(Y)\times_{BG_{\mathbb{F}}}EG_{\mathbb{F}})_\ell^\wedge)$ in $\mathcal{H}(\hat{\mathcal{S}}_{G_{\mathbb{F}}})$.
Note that our rationality condition on $Y$ is always satisfied after a base extension along a sufficiently large unramified $p$-extension $k^\prime/k$.
\end{sect}

\begin{rem}\label{rem: pullback map without homotopy}
Again, suppose each rational component $Y_i$ of the reduced special fibre $Y/\mathbb{F}$ contains a smooth $\mathbb{F}_i$-rational point with $Y_i$ birational to $\tilde{Y}_i \cong \mathbb{P}_{\mathbb{F}_i}^1$.
Let $r$ be a section of $\pi_1^{(\ell)}(Y/\mathbb{F})$.
Further, let $j\colon Y^\circ \hookrightarrow Y$ be the open embedding constructed in \ref{para: nice section} and let $f_\sigma$ and $f_\sigma^\circ$ be the structural morphisms of $Y/\mathbb{F}$ and $Y^\circ/\mathbb{F}$.
Since $Y^\circ$ is a $K(\pi,1)$ with geometrically $\ell$-good fundamental group $\pi_1(Y)$, its \'{e}tale cohomology is given by the cohomology of the profinite group $\pi_1(Y)$.
In particular, purely in terms of \'{e}tale cohomology and cohomology of profinite groups, $r$ induces pullback maps $\mathbb{Z}/\ell^n \rightarrow \mathbb{R}f_{\sigma,\ast}^\circ\mathbb{Z}/\ell^n$ and $r^\ast\colon \mathbb{Z}/\ell^n \rightarrow \mathbb{R}f_{\sigma,\ast}\mathbb{Z}/\ell^n$ in $\mathcal{D}^+(\underline{\rm Mod}_{G_{\mathbb{F}}})$.
\textit{A posteriori}, by Cor.~\ref{cor: unramified on geometric l completion and ho rat pts}, $r^\ast$ is independent from the choice of the smooth rational points $\infty_i\in \tilde{Y}_i(\mathbb{F})$ for $i\in \mathcal{R}$.
In particular, under our rationality conditions on $Y/\mathbb{F}$, we get a definition of a canonical cycle class ${\rm cl}_{r} \in {\rm H}^2(\mathfrak{X},\boldsymbol{\mu}_n)$ of $r$, purely in terms of \'{e}tale cohomology and cohomology of profinite groups.
\end{rem}

\subsection*{Admissible pullback maps.}
By Cor.~\ref{cor: unramified on geometric l completion and ho rat pts}, for any section $s$ of $\pi_1(X/k)$ there a unique quasi-specialized homotopy fixed point $\bar{s}_\ell$ of $(\widehat{\rm Et}(Y/\mathbb{F})\times_{BG_{\mathbb{F}}}EG_{\mathbb{F}})_{\ell}^\wedge$.
We need to show that the induced pullback map is admissible in the sense of \ref{para: admissible pull back map}.

\begin{sect}\label{para: not quite ramification map}
Define the map $r_s^{(\ell)}\colon BG_k\times_{BG_{\mathbb{F}}}EG_{\mathbb{F}} \rightarrow (\widehat{\rm Et}(Y/\mathbb{F})\times_{BG_{\mathbb{F}}}EG_{\mathbb{F}})_\ell^\wedge$ as the composition
\begin{equation*}
 r_s^{(\ell)}\colon
 \xymatrix{
  BG_k\times_{BG_{\mathbb{F}}}EG_{\mathbb{F}} \ar[r]^-s \ar@/_1pc/[rr] &
  \widehat{\rm Et}(X/k)\times_{BG_{\mathbb{F}}}EG_{\mathbb{F}} \ar[r]^-{\rm sp} &
  (\widehat{\rm Et}(Y/\mathbb{F})\times_{BG_{\mathbb{F}}}EG_{\mathbb{F}})_\ell^\wedge
 }.
\end{equation*}
Although on fundamental groups, it induces the geometrically pro-$\ell$ completed ramification map ${\rm ram}_s$ (cf.~\ref{para: quasi-specialized ho rat points} -- in particular, it is trivial on fundamental groups by \cite[Prop.~91]{Stix12}), $r_s^{(\ell)}$ should not be mistaken as an analogue of ${\rm ram}_s$ for homotopy types. 
\end{sect}

\begin{sect}\label{para: splitting of the fundamental groupoid vs not quite ramification map}
Suppose $t\colon B\Pi((\widehat{\rm Et}(Y/\mathbb{F})\times_{BG_{\mathbb{F}}}EG_{\mathbb{F}})_\ell^\wedge) \rightarrow (\widehat{\rm Et}(Y/\mathbb{F})\times_{BG_{\mathbb{F}}}EG_{\mathbb{F}})_\ell^\wedge$ is a splitting in $\mathcal{H}(\hat{\mathcal{S}}_{G_{\mathbb{F}}})$ of the canonical map with the following additional property:
Assume that on cohomology cochains with coefficients $\mathbb{Z}/{\ell^n}$, $r_s^{(\ell)}$ factors through $t$, i.e., assume that there is a map $\varphi$, making the triangle
\begin{equation}\label{eq: splitting of the fundamental groupoid vs not quite ramification map}
 \xymatrix{
  C^\bullet(\widehat{\rm Et}(Y/\mathbb{F})\times_{BG_{\mathbb{F}}}EG_{\mathbb{F}},\mathbb{Z}/{\ell^n}) \ar[rr]^-{r_s^{(\ell),\ast}} \ar[d]^-{t^\ast} &&
  C^\bullet(BG_k\times_{BG_{\mathbb{F}}}EG_{\mathbb{F}},\mathbb{Z}/{\ell^n})
 \\
  C^\bullet(B\Pi(\widehat{\rm Et}(Y/\mathbb{F})\times_{BG_{\mathbb{F}}}EG_{\mathbb{F}}),\mathbb{Z}/{\ell^n}) \ar[urr]_-{\exists \varphi} &&
 }
\end{equation}
commutative in $\mathcal{D}^+(\underline{\rm Mod}_{G_{\mathbb{F}}})$.
As $t$ is a splitting 
and $\bar{s}_\ell$ is induced via $t$ by a specialization of our original section $s$ (cf.~Cor.~\ref{cor: splitting of the fundamental groupoid}), the triangle
\begin{equation*}
 \xymatrix{
  C^\bullet(B\Pi(\widehat{\rm Et}(Y/\mathbb{F})\times_{BG_{\mathbb{F}}}EG_{\mathbb{F}}),\mathbb{Z}/{\ell^n}) \ar[rr]^-\varphi \ar[d]^-{\bar{s}_\ell^\ast} &&
  C^\bullet(BG_k\times_{BG_{\mathbb{F}}}EG_{\mathbb{F}},\mathbb{Z}/{\ell^n})
 \\
 \mathbb{Z}/{\ell^n} \simeq C^\bullet(EG_{\mathbb{F}},\mathbb{Z}/{\ell^n}) \ar[urr]_-{\rm can} &&
 }
\end{equation*}
commutes in $\mathcal{D}^+(\underline{\rm Mod}_{G_{\mathbb{F}}})$, as well.
Piecing together these two commutative triangles and unravelling the definitions, we end up with the commuting square (\ref{eq: specialization in cohomological settings}).
In particular, the pullback map $\bar{s}_\ell^\ast$ is admissible.
\end{sect}

%

We want to use \ref{para: splitting of the fundamental groupoid vs not quite ramification map} to show that $\bar{s}_\ell$ has an admissible pullback map.
As a first step, let us treat models $\mathfrak{X}/\mathfrak{o}$ whose special fibre satisfies a slightly stronger rationality assumption then the one in \ref{para: nice section}:

\begin{prop}\label{prop: specialized homotopy rational point pro l}
Let $X/k$ be a geometrically connected smooth projective curve of positive genus over a $p$-adic field $k$ and $\mathfrak{X}/\mathfrak{o}$ a regular proper flat model.
Suppose each rational component of the reduced special fibre $Y/\mathbb{F}$ contains a smooth $\mathbb{F}$-rational point.
Then for any section $s$ of $\pi_1(X/k)$ there is a unique quasi-specialized homotopy fixed point $\bar{s}_\ell$ of $(\widehat{\rm Et}(Y/\mathbb{F})\times_{BG_{\mathbb{F}}}EG_{\mathbb{F}})_{\ell}^\wedge$ in $\mathcal{H}(\hat{\mathcal{S}}_{/ BG_{\mathbb{F}}})$ inducing the commutative diagram 
\begin{equation}\label{eq: specialized homotopy rational point pro l}
\xymatrix{
  C^\bullet(\widehat{\rm Et}(Y/\mathbb{F})\times_{BG_{\mathbb{F}}}EG_{\mathbb{F}},\Lambda) \ar[r]^-{\bar{s}_\ell^\ast} \ar[d]^-{{\rm sp}^\ast} &
  C^\bullet(EG_{\mathbb{F}},\Lambda)\simeq \Lambda \ar[d]^-{\rm can}
 \\
  C^\bullet(\widehat{\rm Et}(X/k)\times_{BG_{\mathbb{F}}}EG_{\mathbb{F}},\Lambda) \ar[r]^-{s^\ast} &
  C^\bullet(BG_k\times_{BG_{\mathbb{F}}}EG_{\mathbb{F}},\Lambda)
 }
\end{equation}
in the derived category $\mathcal{D}^+(\underline{\rm Mod}_{G_{\mathbb{F}}})$ for $\Lambda$ any continuous finite $\ell$-torsion $G_{\mathbb{F}}$-module.
\end{prop}

\begin{proof}
Let us first reformulate: We claim that the square
\begin{equation*}
 \xymatrix{
  \mathbb{R}f_{\sigma,\ast}\Lambda \ar[rr]^-{\bar{s}_\ell^\ast} \ar[d]^-{{\rm can}^{-1}} &&
  \Lambda \ar[dd]^-{\sigma^\ast({\rm can})}
 \\
  \sigma^\ast\mathbb{R}f_\ast\Lambda \ar[d]^-{\sigma^\ast({\rm can})} &&
 \\
  \sigma^\ast\mathbb{R}\eta_\ast\mathbb{R}f_{\eta,\ast}\Lambda \ar[rr]^-{\sigma^\ast\mathbb{R}\eta_\ast(s^\ast)} &&
  \sigma^\ast\mathbb{R}\eta_\ast\Lambda
 }
\end{equation*}
is commutative in $\mathcal{D}^+(\underline{\rm Mod}_{G_{\mathbb{F}}})$ (cf.~\ref{para: definition of the cycle class, reformulation II}).
Using the projection formula, it suffices to treat the case $\Lambda = \mathbb{Z}/\ell^n$.
We want to use the arguments in \ref{para: splitting of the fundamental groupoid vs not quite ramification map}.
The map $r_s^{(\ell),\ast}$ corresponds to the composition $\mathbb{R}f_{\sigma,\ast}\mathbb{Z}/\ell^n \rightarrow \sigma^\ast\mathbb{R}\eta_\ast\mathbb{Z}/\ell^n$ around the lower left vertex of the square.
For the splitting $t\colon B\Pi((\widehat{\rm Et}(Y/\mathbb{F})\times_{BG_{\mathbb{F}}}EG_{\mathbb{F}})_\ell^\wedge) \rightarrow (\widehat{\rm Et}(Y/\mathbb{F})\times_{BG_{\mathbb{F}}}EG_{\mathbb{F}})_\ell^\wedge$ we take the map $t_j$ induced by the open subscheme $j\colon Y^\circ = Y \setminus \{ \infty_i ~\vert~ i\in \mathcal{R} \} \hookrightarrow Y$, for $\mathcal{R}$ the set of indices $i$ with $\tilde{Y}_i$ a rational component of the normalization $\tilde{Y} \rightarrow Y$ and $\infty_i$ an $\mathbb{F}$-rational point in $\tilde{Y}_i$ dominating a smooth point of $Y$ (cf.~\ref{para: nice section}).
Let $f_\sigma^\circ$ be the structural morphism of $Y^\circ/\mathbb{F}$.
By the discussion in \ref{para: splitting of the fundamental groupoid vs not quite ramification map} and our translation, it remains to show that the composition around the lower left vertex of the square $r_s^{(\ell),\ast}$ factors in $\mathcal{D}^+(\underline{\rm Mod}_{G_{\mathbb{F}}})$ through the canonical map $\mathbb{R}f_{\sigma,\ast}\mathbb{Z}/\ell^n \rightarrow \mathbb{R}f_{\sigma,\ast}^\circ\mathbb{Z}/\ell^n$.
\\
Let $i\colon Z \hookrightarrow Y$ be the reduced complement of $Y^\circ$ in $Y$.
Then $Z \cong {\rm Spec}(\mathbb{F})\otimes \mathcal{R}$ (the coproduct of $\mathcal{R}$-many copies of ${\rm Spec}(\mathbb{F})$) is regular and $i$ factors through the regular locus $j_{\rm reg}\colon Y_{\rm reg}\hookrightarrow Y$, say via $i_{\rm reg}\colon Z \hookrightarrow Y_{\rm reg}$.
From relative purity we get
\begin{equation*}
 \mathbb{R}i^!\mathbb{Z}/\ell^n = \mathbb{R}i_{\rm reg}^!j_{\rm reg}^\ast\mathbb{Z}/\ell^n = \mathbb{Z}/\ell^n(-1)[-2],
\end{equation*}
i.e., we get the exact Gysin triangle in $\mathcal{D}^+(\underline{\rm Mod}_{G_{\mathbb{F}}})$
\begin{equation*}
 \xymatrix{
  \bigoplus_{i\in\mathcal{R}}\mathbb{Z}/\ell^n(-1)[-2] \ar[r] &
  \mathbb{R}f_{\sigma,\ast} \mathbb{Z}/\ell^n \ar[r] &
  \mathbb{R}f_{\sigma,\ast}^\circ\mathbb{Z}/\ell^n \ar[r]^-{+1} &
  \bigoplus_{\mathcal{R}}\mathbb{Z}/\ell^n(-1)[-1]
 }
\end{equation*}
In particular, $r_s^{(\ell),\ast}$ factors through $\mathbb{R}f_{\sigma,\ast}\mathbb{Z}/\ell^n \rightarrow \mathbb{R}f_{\sigma,\ast}^\circ\mathbb{Z}/\ell^n$ if and only if $r_s^{(\ell),\ast}$ restricts to the trivial map $\mathbb{Z}/\ell^n(-1)[-2] \rightarrow \sigma^\ast\mathbb{R}\eta_\ast \mathbb{Z}/\ell^n$ for each $i\in \mathcal{R}$.
\\
To see this, first note that the $i^{\rm th}$ map $\mathbb{Z}/\ell^n(-1)[-2] \rightarrow \mathbb{R}f_{\sigma,\ast}\mathbb{Z}/\ell^n$ corresponds to the first Chern class
\begin{equation*}
 \hat{c}_1[\mathcal{O}_Y(\infty_i)] \in {\rm H}^2(Y,\boldsymbol{\mu}_{\ell^n}) = [\mathbb{Z}/\ell^n(-1)[-2], \mathbb{R}f_{\sigma,\ast}\mathbb{Z}/\ell^n]_{\underline{\rm Mod}_{\mathbb{G}_{\mathbb{F}}}}.
\end{equation*}
It follows, that the restriction to the $i^{\rm th}$ component $\mathbb{Z}/\ell^n(-1)[-2] \rightarrow \sigma^\ast\mathbb{R}\eta_\ast \mathbb{Z}/\ell^n \simeq \mathbb{R}\eta_{{\rm fet},\ast} \mathbb{Z}/\ell^n$ of the map $r_s^{(\ell),\ast}$ corresponds to the class
\begin{equation*}
 s^\ast\hat{c}_1[\mathcal{L}] \in {\rm H}^2(G_k,\boldsymbol{\mu}_{\ell^n}) = [\mathbb{Z}/\ell^n(-1)[-2], \mathbb{R}\eta_{{\rm fet},\ast}\mathbb{Z}/\ell^n]_{\underline{\rm Mod}_{\mathbb{G}_{\mathbb{F}}}},
\end{equation*}
where $\mathcal{L}$ is the generic fibre of a line bundle mapping to $\mathcal{O}(\infty_i)$ under the epimorphism of Picard groups ${\rm Pic}(\mathfrak{X}) \twoheadrightarrow {\rm Pic}(Y)$.
For the latter, see \cite[Exp.~XIII Prop.~3.2]{SGA4.3}.
From the characterization of ${\rm cl}_s$ via duality (see \cite[Sect.~6.1]{Stix12}) we get $s^\ast\hat{c}_1[\mathcal{L}] = {\rm cl}_s\cup \hat{c}_1[\mathcal{L}]$.
Finally, via Tate-Lichtenbaum duality $s^\ast\hat{c}_1[\mathcal{L}] = 0$ follows from the algebraicity of ${\rm cl}_s$ in ${\rm H}^2(X,\mathbb{Z}_\ell(1))$ (see \cite[Cor.~3.4 and Rem.~3.7 (ii)]{EsnaultWittenberg09}), which finishes the proof.
\end{proof}


With a little bit more work, we can generalize Prop.\ \ref{prop: specialized homotopy rational point pro l} to arbitrary regular proper flat models $\mathfrak{X}/\mathfrak{o}$:

\begin{thm}\label{thm: specialized homotopy rational point pro l}
Let $X/k$ be a geometrically connected smooth projective curve of positive genus over a $p$-adic field $k$ and $\mathfrak{X}/\mathfrak{o}$ a regular proper flat model.
Let $\bar{s}_\ell$ be the unique quasi-specialized homotopy fixed point of $(\widehat{\rm Et}(Y/\mathbb{F})\times_{BG_{\mathbb{F}}}EG_{\mathbb{F}})_{\ell}^\wedge$ in $\mathcal{H}(\hat{\mathcal{S}}_{/ BG_{\mathbb{F}}})$ given by a section $s$ of $\pi_1(X/k)$ (cf.~Cor.~\ref{cor: unramified on geometric l completion and ho rat pts}).
Then $\bar{s}_\ell$ has an admissible pullback map.
More generally, for any continuous finite $\ell$-torsion $G_{\mathbb{F}}$-module $\Lambda$, the induced diagram \eqref{eq: specialized homotopy rational point pro l} in the derived category $\mathcal{D}^+(\underline{\rm Mod}_{G_{\mathbb{F}}})$ commutes.
\end{thm}

\begin{proof}
Consider the difference between the compositions around the lower left and the upper right vertex of \eqref{eq: specialized homotopy rational point pro l}:
\begin{equation*}
 \varphi_{s,\Lambda} := r_s^{(\ell),\ast} - ({\rm can} \circ \bar{s}_\ell^\ast)\colon
 \xymatrix{
  C^\bullet(\widehat{\rm Et}(Y/\mathbb{F})\times_{BG_{\mathbb{F}}}EG_{\mathbb{F}},\Lambda) \ar[r] &
  C^\bullet(BG_k\times_{BG_{\mathbb{F}}}EG_{\mathbb{F}},\Lambda)
 }.
\end{equation*}
Precomposing with $C^\bullet(B\Pi(\widehat{\rm Et}(Y/\mathbb{F})\times_{BG_{\mathbb{F}}}EG_{\mathbb{F}}),\Lambda) \rightarrow C^\bullet(\widehat{\rm Et}(Y/\mathbb{F})\times_{BG_{\mathbb{F}}}EG_{\mathbb{F}},\Lambda)$ trivializes $\varphi_{s,\Lambda}$, since $\pi_1(\bar{s}_\ell)$ is a specialization of our original section $s$.
We have to show that $\varphi_{s,\Lambda}$ itself is trivial.
\\
Let us reformulate this using the notation in the proof of Prop.\ \ref{prop: specialized homotopy rational point pro l}:
$\varphi_{s,\Lambda}$ corresponds to a morphism $\mathbb{R}f_{\sigma,\ast}\Lambda \rightarrow \sigma^\ast\mathbb{R}\eta_\ast\Lambda \simeq \mathbb{R}\eta_{{\rm fet},\ast}\Lambda$.
Again, by the projection formula we may assume $\Lambda =\mathbb{Z}/\ell^n$.
Denote the corresponding morphism by $\varphi_{s,n}$.
By Lem.~\ref{lem:etale vs finite etale cohomology} there is an exact triangle
\begin{equation*}
 \xymatrix{
  \mathbb{R} f_{\sigma,{\rm fet},\ast}\mathbb{Z}/\ell^n \ar[r]^-{\gamma^\ast} &
  \mathbb{R} f_{\sigma,\ast}\mathbb{Z}/\ell^n \ar[r] &
  \bigoplus_{i\in \mathcal{R}}  (\mathbb{R}^2 \tilde{f}_{i,\ast}\mathbb{Z}/\ell^n)[-2] \ar[r] &
  \mathbb{R} f_{\sigma,{\rm fet},\ast}\mathbb{Z}/\ell^n[1]
 }.
\end{equation*}
in $\mathcal{D}^+(\underline{\rm Mod}_{G_{\mathbb{F}}})$.
Here $\tilde{f}_{i}$ denotes the structural map of the rational component $\tilde{Y}_i/\mathbb{F}$. 
Since the composition $\varphi_{s,n} \circ \gamma^\ast$ is trivial, $\varphi_{s,n}$ factors through $\mathbb{R} f_{\sigma,\ast}\mathbb{Z}/\ell^n \rightarrow \bigoplus_{i\in \mathcal{R}}  (\mathbb{R}^2 \tilde{f}_{i,\ast}\mathbb{Z}/\ell^n)[-2]$.
Say, $\varphi_{s,n}$ factors via a map
\begin{equation*}
 \psi_n:
 \xymatrix{
  \bigoplus_{i\in \mathcal{R}}  (\mathbb{R}^2 \tilde{f}_{i,\ast}\mathbb{Z}/\ell^n)[-2] \ar[r] &
  \mathbb{R}\eta_{{\rm fet},\ast}\mathbb{Z}/\ell^n
 }.
\end{equation*}
We may assume that the maps $\psi_n$ are compatible for various $n$:
By construction, the maps $\varphi_{s,n}$ are compatible for various $n$. 
Thus, working with $\ell$-adic sheaves instead, we even get a map $\psi: \bigoplus_{i\in \mathcal{R}}  (\mathbb{R}^2 \tilde{f}_{i,\ast}\mathbb{Z}_\ell)[-2] \rightarrow \mathbb{R}\eta_{{\rm fet},\ast}\mathbb{Z}_\ell$ inducing compatible maps $\psi_n$ modulo $\ell^n$, still factoring $\varphi_{s,n}$.
It suffices to show that all the $\psi_n$ are trivial, i.e., $\psi = 0$.
\\
Consider the component $(\mathbb{R}^2 \tilde{f}_{i,\ast}\mathbb{Z}/\ell^n)[-2]$ for $i \in \mathcal{R}$.
Since finite fields have trivial Brauer groups, $\tilde{Y}_i$ is isomorphic to $\mathbb{P}_{\mathbb{F}_i}^1$ for a finite extension $\mathbb{F}_i/\mathbb{F}$.
In particular, we get
\begin{equation*}
 {\rm res}_{G_{\mathbb{F}_i}}^{G_{\mathbb{F}}}\mathbb{R}^2 \tilde{f}_{i,\ast}\mathbb{Z}/\ell^n \cong
 \bigoplus_{G_{\mathbb{F}}/G_{\mathbb{F}_i}} \mathbb{Z}/\ell^n(-1) \cong
 {\rm res}_{G_{\mathbb{F}_i}}^{G_{\mathbb{F}}} {\rm ind}_{G_{\mathbb{F}_i}}^{G_{\mathbb{F}}} \mathbb{Z}/\ell^n(-1)
\end{equation*}
for the restriction.
Thus, the counit ${\rm res}_{G_{\mathbb{F}_i}}^{G_{\mathbb{F}}} {\rm ind}_{G_{\mathbb{F}_i}}^{G_{\mathbb{F}}} \mathbb{Z}/\ell^n(-1) \rightarrow \mathbb{Z}/\ell^n(-1)$ induces an isomorphism
\begin{equation*}
 \mathbb{R}^2 \tilde{f}_{i,\ast}\mathbb{Z}/\ell^n \cong {\rm ind}_{G_{\mathbb{F}_i}}^{G_{\mathbb{F}}}  \mathbb{Z}/\ell^n(-1).
\end{equation*}
Choose a finite unramified extension $k^\prime / k$ s.t.~each rational component of the reduced special fibre $Y^\prime/\mathbb{F}^\prime$ of $\mathfrak{X}^\prime = \mathfrak{X}\otimes_{\mathfrak{o}}\mathfrak{o}^\prime/\mathfrak{o}^\prime$ contains a smooth $\mathbb{F}^\prime$-rational point.
Here $\mathfrak{o}^\prime$ denotes the ring of integers of $k^\prime$.
In particular, $\mathbb{F}^\prime/\mathbb{F}$ dominates all extensions $\mathbb{F}_i/\mathbb{F}$.
Let $f^\prime$ be the structural morphism of $\mathfrak{X}^\prime / \mathfrak{o}^\prime$.
We get $\mathbb{R}f_{\sigma^\prime,({\rm fet}),\ast}^\prime \mathbb{Z}/\ell^n = {\rm res}_{G_{\mathbb{F}^\prime}}^{G_{\mathbb{F}}} (\mathbb{R}f_{\sigma,({\rm fet}),\ast} \mathbb{Z}/\ell^n)$, $\mathbb{R}\eta_{{\rm fet},\ast}^\prime\mathbb{Z}/\ell^n = {\rm res}_{G_{\mathbb{F}^\prime}}^{G_{\mathbb{F}}} (\mathbb{R}\eta_{{\rm fet},\ast}\mathbb{Z}/\ell^n)$ ($\eta^\prime$ the generic and $\sigma^\prime$ the closed point of $\mathfrak{o}^\prime$) and the restriction of the above exact triangle to $\mathcal{D}^+(\underline{\rm Mod}_{G_{\mathbb{F}^\prime}})$ is the triangle
\begin{equation}\label{eq: restricted triangle}
 \xymatrix{
  \mathbb{R} f_{\sigma^\prime,{\rm fet},\ast}^\prime\mathbb{Z}/\ell^n \ar[r]^-{\gamma^{\prime,\ast}} &
  \mathbb{R} f_{\sigma^\prime,\ast}^\prime\mathbb{Z}/\ell^n \ar[r] &
  \bigoplus_{i\in \mathcal{R}^\prime}  \mathbb{Z}/\ell^n(-1)[-2] \ar[r] &
  \mathbb{R} f_{\sigma^\prime,{\rm fet},\ast}^\prime\mathbb{Z}/\ell^n[1]
 }
\end{equation}
given by Lem.~\ref{lem:etale vs finite etale cohomology} applied to $Y^\prime/\mathbb{F}^\prime$ directly.
Here we used $\mathbb{F}_i \subseteq \mathbb{F}^\prime$ and $\mathcal{R}^\prime$ denotes the set of rational components of $Y^\prime$, i.e., $G_{\mathbb{F}}$ acts on $\mathcal{R}^\prime$ with $\mathcal{R}^\prime/G_{\mathbb{F}} = \mathcal{R}$.
\\
We claim that it is enough to show that ${\rm res}_{G_{\mathbb{F}^\prime}}^{G_{\mathbb{F}}} \psi_n$ is trivial for all $n$.
We can check this separately for each component $(\mathbb{R}^2 \tilde{f}_{i,\ast}\mathbb{Z}/\ell^n)[-2]$, $i \in \mathcal{R}$ and after twisting by $(-)(1)$ and shifting by $(-)[2]$.
We have
\begin{align*}
  [\mathbb{R}^2 \tilde{f}_{i,\ast}\mathbb{Z}/\ell^n(1), \mathbb{R}\eta_{{\rm fet},\ast}\mathbb{Z}/\ell^n(1)[2]]_{\underline{\rm Mod}_{G_{\mathbb{F}}}} & =
  [{\rm sp}^\ast{\rm ind}_{G_{\mathbb{F}_i}}^{G_{\mathbb{F}}}  \mathbb{Z}/\ell^n, \mathbb{Z}/\ell^n(1)[2]]_{\underline{\rm Mod}_{G_{k}}}
 \\
  & =
  [{\rm ind}_{G_{k_i}}^{G_{k}}  \mathbb{Z}/\ell^n, \mathbb{Z}/\ell^n(1)[2]]_{\underline{\rm Mod}_{G_{k}}}
 \\
  & =
  [\mathbb{Z}/\ell^n, \mathbb{Z}/\ell^n(1)[2]]_{\underline{\rm Mod}_{G_{k_i}}}
 \\
  & =
  {\rm H}^2(G_{k_i},\mathbb{Z}/\ell^n(1))
\end{align*}
for ${\rm sp}\colon G_k\rightarrow G_{\mathbb{F}}$ the specialization map and $k_i/k$ the unramified extension with residue extension $\mathbb{F}_i/\mathbb{F}$.
Similarly one gets
\begin{align*}
  [{\rm res}_{G_{\mathbb{F}_i}}^{G_{\mathbb{F}}} \mathbb{R}^2 \tilde{f}_{i,\ast}\mathbb{Z}/\ell^n(1), {\rm res}_{G_{\mathbb{F}_i}}^{G_{\mathbb{F}}} \mathbb{R}\eta_{{\rm fet},\ast}\mathbb{Z}/\ell^n(1)[2]]_{\underline{\rm Mod}_{G_{\mathbb{F}_i}}} & =
 \\
  [{\rm res}_{G_{k_i}}^{G_{k}} {\rm ind}_{G_{k_i}}^{G_{k}}  \mathbb{Z}/\ell^n, \mathbb{Z}/\ell^n(1)[2]]_{\underline{\rm Mod}_{G_{k_i}}} & =
  \bigoplus_{G_{\mathbb{F}}/G_{\mathbb{F}_i}} {\rm H}^2(G_{k_i},\mathbb{Z}/\ell^n(1))
\end{align*}
where the restriction map ${\rm res}_{G_{\mathbb{F}_i}}^{G_{\mathbb{F}}}(-)$ between the left hand sides corresponds to the diagonal embedding ${\rm H}^2(G_{k_i},\mathbb{Z}/\ell^n(1)) \hookrightarrow \bigoplus_{G_{\mathbb{F}}/G_{\mathbb{F}_i}} {\rm H}^2(G_{k_i},\mathbb{Z}/\ell^n(1))$ on the right hand sides.
Finally,
\begin{align*}
  [{\rm res}_{G_{\mathbb{F}^\prime}}^{G_{\mathbb{F}}} \mathbb{R}^2 \tilde{f}_{i,\ast}\mathbb{Z}/\ell^n(1), {\rm res}_{G_{\mathbb{F}^\prime}}^{G_{\mathbb{F}}} \mathbb{R}\eta_{{\rm fet},\ast}\mathbb{Z}/\ell^n(1)[2]]_{\underline{\rm Mod}_{G_{\mathbb{F}_i}}} & =
 \\
  [{\rm res}_{G_{k^\prime}}^{G_{k}} {\rm ind}_{G_{k_i}}^{G_{k}}  \mathbb{Z}/\ell^n, \mathbb{Z}/\ell^n(1)[2]]_{\underline{\rm Mod}_{G_{k^\prime}}} & =
  \bigoplus_{G_{\mathbb{F}}/G_{\mathbb{F}_i}} {\rm H}^2(G_{k^\prime},\mathbb{Z}/\ell^n(1))
\end{align*}
holds and the restriction map ${\rm res}_{G_{\mathbb{F}^\prime}}^{G_{\mathbb{F}_i}}(-)$ between the left hand sides corresponds to the multiplication by the degree $[k^\prime:k_i]$ on the right hand sides.
For the latter, note that both Galois cohomology groups ${\rm H}^2(G_{k_i},\mathbb{Z}/\ell^n(1))$ and ${\rm H}^2(G_{k^\prime},\mathbb{Z}/\ell^n(1))$ are canonically isomorphic to $\frac{1}{\ell^n}\mathbb{Z} /\mathbb{Z}$.
Summing up, the composed restriction map ${\rm res}_{G_{\mathbb{F}^\prime}}^{G_{\mathbb{F}_i}}(-)\colon$
\begin{equation*}
 \xymatrix{
  [(\mathbb{R}^2 \tilde{f}_{i,\ast}\mathbb{Z}/\ell^n)[2], \mathbb{R}\eta_{{\rm fet},\ast}\mathbb{Z}/\ell^n]_{\underline{\rm Mod}_{G_{\mathbb{F}}}} \ar[r]&
  [{\rm res}_{G_{\mathbb{F}^\prime}}^{G_{\mathbb{F}}} (\mathbb{R}^2 \tilde{f}_{i,\ast}\mathbb{Z}/\ell^n)[2], \mathbb{R}\eta_{{\rm fet},\ast}^\prime\mathbb{Z}/\ell^n]_{\underline{\rm Mod}_{G_{\mathbb{F}_i}}}
 }
\end{equation*}
corresponds to the diagonal embedding $\frac{1}{\ell^n}\mathbb{Z} /\mathbb{Z} \hookrightarrow \bigoplus_{G_{\mathbb{F}}/G_{\mathbb{F}_i}} \frac{1}{\ell^n}\mathbb{Z} /\mathbb{Z}$ followed by multiplication by $[k^\prime:k_i]$.
Taking the limit over all $n$ we get a monomorphism, so $\psi_n$ is trivial for \emph{all} $n$ if ${\rm res}_{G_{\mathbb{F}^\prime}}^{G_{\mathbb{F}}} \psi_n$ is trivial for \emph{all} $n$. 
\\
We have to show that ${\rm res}_{G_{\mathbb{F}^\prime}}^{G_{\mathbb{F}}} \psi_n$ is trivial.
Let us analyse the restricted triangle \eqref{eq: restricted triangle}:
Using the open embedding $j\colon Y^{\prime,\circ} \hookrightarrow Y^\prime$ (cf.~\ref{para: nice section}), we get a commutative diagram
\begin{equation*}
 \xymatrix{
  \mathbb{R}f_{\sigma^\prime,{\rm fet},\ast}^\prime\mathbb{Z}/\ell^n \ar[r] \ar[d]^-\sim &
  \mathbb{R}f_{\sigma^\prime,\ast}^\prime\mathbb{Z}/\ell^n \phantom{,} \ar[d]
  \\
  \mathbb{R}f_{\sigma^\prime,{\rm fet},\ast}^{\prime,\circ}\mathbb{Z}/\ell^n \ar[r]^-\sim &
  \mathbb{R}f_{\sigma^\prime,\ast}^{\prime,\circ}\mathbb{Z}/\ell^n,
 }
\end{equation*}
i.e., a retraction of $\mathbb{R}f_{\sigma^\prime,{\rm fet},\ast}^\prime\mathbb{Z}/\ell^n \rightarrow \mathbb{R}f_{\sigma^\prime,\ast}^\prime\mathbb{Z}/\ell^n$.
Here $f_{\sigma^\prime}^{\prime,\circ}$ is the structural morphism of $Y^{\prime,\circ}/\mathbb{F}^\prime$.
It follows that $[\mathbb{R}f_{\sigma^\prime,\ast}^\prime\mathbb{Z}/\ell^n[1],-]_{\underline{\rm Mod}_{G_{\mathbb{F}^\prime}}} \rightarrow [\mathbb{R}f_{\sigma^\prime,{\rm fet},\ast}^\prime\mathbb{Z}/\ell^n[1],-]_{\underline{\rm Mod}_{G_{\mathbb{F}^\prime}}}$ is split-surjective.
In particular,
\begin{equation*}
 \xymatrix{
  \bigoplus_{i\in \mathcal{R}^\prime} [\mathbb{Z}/\ell^n(-1)[-2],\mathbb{R}\eta_{{\rm fet},\ast}^\prime\mathbb{Z}/\ell^n]_{\underline{\rm Mod}_{G_{\mathbb{F}^\prime}}} \ar@{^(->}[r] &
  [\mathbb{R}f_{\sigma^\prime,\ast}^\prime\mathbb{Z}/\ell^n,\mathbb{R}\eta_{{\rm fet},\ast}^\prime\mathbb{Z}/\ell^n]_{\underline{\rm Mod}_{G_{\mathbb{F}^\prime}}}
 }
\end{equation*}
is a monomorphism.
Let $s^\prime = s\vert_{G_{k^\prime}}$ be the restricted section of $\pi_1(X^\prime/k^\prime)$, for $X^\prime/k^\prime$ the generic fibre of $\mathfrak{X}^\prime/\mathfrak{o}^\prime$.
By Prop.~\ref{prop: specialized homotopy rational point pro l}, ${\rm res}_{G_{\mathbb{F}^\prime}}^{G_{\mathbb{F}}} (\varphi_{s,n}) = \varphi_{s^\prime,n}$ is trivial, forcing ${\rm res}_{G_{\mathbb{F}^\prime}}^{G_{\mathbb{F}}} (\psi_n)$ to be trivial and Thm.~\ref{thm: specialized homotopy rational point pro l} follows.
\end{proof}

\subsection*{A canonical lift of $\boldsymbol{{\rm cl}_s}$.}
Now that we have collected all the ingredients in a canonical way (cf.~Cor.~\ref{cor: unramified on geometric l completion and ho rat pts} and Thm.~\ref{thm: specialized homotopy rational point pro l}), we can finally execute our general recipe for a canonical lift of the cycle class ${\rm cl}_s$:

\begin{thm}\label{thm: canonical lift of the cycle class to the model}
Let $X/k$ be a geometrically connected smooth projective curve of positive genus over a $p$-adic field $k$ and $\mathfrak{X}/\mathfrak{o}$ a regular proper flat model.
Then for any $\ell \neq p$ and any section $s$ of $\pi_1(X/k)$, the induced $\ell$-adic cycle class ${\rm cl}_s$ admits a canonical lift ${\rm cl}_s^{\mathfrak{X}}$ to ${\rm H}^2(\mathfrak{X},\boldsymbol{\mu}_{\ell^n})$, compatible for various $n$.
Further, the class ${\rm cl}_s^{\mathfrak{X}}$ is natural in the pair $(\mathfrak{X}/\mathfrak{o},s)$.
\end{thm}

\begin{proof}
By Cor.~\ref{cor: unramified on geometric l completion and ho rat pts}, each section $s$ of $\pi_1(Y/\mathbb{F})$ induces a unique quasi-specialized homotopy fixed point $\bar{s}_\ell$ of the pro-$\ell$ completion $(\widehat{\rm Et}(Y/\mathbb{F})\times_{BG_{\mathbb{F}}}EG_{\mathbb{F}})_{\ell}^\wedge$ in $\mathcal{H}(\hat{\mathcal{S}}_{/ BG_{\mathbb{F}}})$.
By Thm.~\ref{thm: specialized homotopy rational point pro l}, the induced pullback map $\bar{s}_\ell^\ast\colon C^\bullet(\widehat{\rm Et}(Y/\mathbb{F})\times_{BG_{\mathbb{F}}}EG_{\mathbb{F}},\mathbb{Z}/\ell^n) \rightarrow \mathbb{Z}/\ell^n$ is admissible, i.e., compatible with the pullback map induced by $s$ via the specialization maps.
Thus, by Cor.~\ref{cor: lift of the cycle class}, its canonical cycle class ${\rm cl}_{\bar{s}_\ell} \in {\rm H}^2(\mathfrak{X},\boldsymbol{\mu}_{\ell^n})$ lifts the cycle class ${\rm cl}_s$ and we set ${\rm cl}_s^{\mathfrak{X}} :=  {\rm cl}_{\bar{s}_\ell}$.
Finally, the functoriality of ${\rm cl}_s^{\mathfrak{X}}$ in $(\mathfrak{X}/\mathfrak{o},s)$ holds by construction.
\end{proof}

\begin{rem}\label{rem: construction of the cycle class without homotopy}
The class ${\rm cl}_s^{\mathfrak{X}}$ can be defined purely in terms of \'{e}tale cohomology and cohomology of profinite groups: 
For a sufficiently large unramified $p$-extension $k^\prime/k$ (with Galois group $G$, normalization $\mathfrak{o}^\prime/\mathfrak{o}$ and residue extension $\mathbb{F}^\prime/\mathbb{F}$), the rationality condition of \ref{para: nice section} holds for $Y\otimes_{\mathbb{F}}\mathbb{F}^\prime$.
The section $s$ restricts to a section $s^\prime$  of $\pi_1(X\otimes_kk^\prime/k^\prime)$.
Arguing as in proof of Thm.~\ref{thm: canonical lift of the cycle class to the model} using Rem.~\ref{rem: pullback map without homotopy} gives the canonical lift ${\rm cl}_{s^\prime}^{\mathfrak{X}\otimes_{\mathfrak{o}}\mathfrak{o}^\prime}$ of ${\rm cl}_{s^\prime}$ to ${\rm H}^2(\mathfrak{X} \otimes_{\mathfrak{o}}\mathfrak{o}^\prime,\boldsymbol{\mu}_{\ell^n})^G$ purely in terms of \'{e}tale cohomology and cohomology of profinite groups. 
Here the $G$-invariance holds by \ref{para: cycle class of a ho rat point of the special fibre first properties}.
The order of $G$ kills all higher cohomology groups ${\rm H}^q(G,\Lambda)$.
Since this order is a power of $p\neq \ell$, these cohomology groups vanish for $\Lambda$ a finite $\ell$-torsion $G$-module.
In particular, ${\rm H}^2(\mathfrak{X},\boldsymbol{\mu}_{\ell^n}) = {\rm H}^2(\mathfrak{X} \otimes_{\mathfrak{o}}\mathfrak{o}^\prime,\boldsymbol{\mu}_{\ell^n})^G$ and ${\rm H}^2(X,\boldsymbol{\mu}_{\ell^n}) = {\rm H}^2(X \otimes_kk^\prime,\boldsymbol{\mu}_{\ell^n})^G$ follow using the respective Hochschild-Serre spectral sequences.
Thus, ${\rm cl}_{s^\prime}^{\mathfrak{X} \otimes_{\mathfrak{o}}\mathfrak{o}^\prime}$ induces a class in ${\rm H}^2(\mathfrak{X},\boldsymbol{\mu}_{\ell^n})$.
By construction, this class lifts ${\rm cl}_s$ and coincides with the class ${\rm cl}_s^{\mathfrak{X}}$, but is defined purely in terms of \'{e}tale cohomology and cohomology of profinite groups.
The drawback of this construction is that its canonicity comes only a posteriori via Thm.~\ref{thm: canonical lift of the cycle class to the model}.
\end{rem}



\begin{thebibliography}{SK}

\bibitem[AB69]{AndreottiBombieri}
A.\ Andreotti and E.\ Bombieri.
\newblock {Sugli omeomorfismi delle variet\`{a} algebriche}.
\newblock {\em Annali della Scuola Normale Superiore di Pisa, Classe di Scienze
  $3^{e}$ s\'{e}rie}, 23(3):\ 431--450, 1969.

  
\bibitem[BK72]{BousfieldKan}
Aldridge~K.\ Bousfield and Daniel~M.\ Kan.
\newblock {\em Homotopy Limits, Completions and Localizations}, volume 304 of
  {\em Lecture Notes in Mathematics}.
\newblock Springer-Verlag, Berlin-Heidelberg-New York, 1972.

\bibitem[Bou89]{Bousfield89}
Aldridge~K.\ Bousfield.
\newblock {Homotopy Spectral Sequences and Obstructions}.
\newblock {\em Israel Journal of Mathematics}, 66(1--3):\ 54--104, 1989.

\bibitem[EW09]{EsnaultWittenberg09}
H\'{e}l\`{e}ne Esnault and Olivier Wittenberg.
\newblock {Remarks on Cycle Classes of Sections of the Arithmetic Fundamental
  Group}.
\newblock {\em Moscow Mathematical Journal}, 9(3):\ 451--467, 2009.

\bibitem[Fri82]{Friedlander}
Eric~M.\ Friedlander.
\newblock {\em Etale Homotopy of Simplicial Schemes}, volume 104 of {\em Annals
  of Mathematics Studies}.
\newblock Princeton University Press, 1982.

\bibitem[Fu11]{Fu11}
Lei Fu.
\newblock {\em Etale Chomology Theory}, volume~13 of {\em Nankai Tracts in
  Mathematics}.
\newblock World Scientific, 2011.

\bibitem[Fuj02]{Fujiwara02}
Kazuhiro Fujiwara.
\newblock {A Proof of the Absolute Purity Conjecture (after Gabber)}.
\newblock In {\em Algebraic Geometry 2000, Azumino}, pages 153--183. 2002.

\bibitem[HS13]{HarpazSchlank13}
Yonatan Harpaz and Tomer M.~Schlank.
\newblock {Homotopy Obstructions to Rational Points}.
\newblock In {\em Torsors, \'{E}tale Homotopy and Applications to Rational Points, LMS Lecture Notes Series 405}, pages 280--413. 2013.

\bibitem[Jan88]{Jannsen88}
Uwe Jannsen.
\newblock {Continuous \'{E}tale Cohomology}.
\newblock {\em Mathematische Annalen}, 280:\ 207--245, 1988.

\bibitem[Koe05]{Koenigsmann05}
Jochen Koenigsmann.
\newblock {On the {\textquoteleft}section conjecture{\textquoteright}~in
  anabelian geometry}.
\newblock {\em Journal f\"{u}r die reine und angewandte Mathematik}, 588:\
  221--235, 2005.

\bibitem[Kol96]{Kollar96}
J\'{a}nos Koll\'{a}r.
\newblock {\em Rational Curves on Algebraic Varieties}, volume~32 of {\em
  Ergebnisse der Mathematik und ihrer Grenzgebiete, 3. Folge}.
\newblock Springer-Verlag, Berlin-Heidelberg-New York, 1996.

\bibitem[Mil04]{Milne04}
James S.~Milne.
\newblock {\em Arithmetic Duality Theorems}, Second Edition.
\newblock Kea Books, 2004.

\bibitem[Moc99]{Mochizuki99}
Shinichi Mochizuki.
\newblock {The local pro-$p$ anabelian geometry of curves}.
\newblock {\em Inventiones Mathematicae}, 138(2):\ 319--423, 1999.

\bibitem[Mor96]{Morel96}
Fabien Morel.
\newblock {Ensembles Profinis Simpliciaux et Interpr\'{e}tation
  G\'{e}om\'{e}trique du Foncteur $T$}.
\newblock {\em Bull.~Soc.~math.~France}, 124:\ 347--373, 1996.

\bibitem[Par90]{Parshin90}
Alexei~Nikolajewitsch Parschin.
\newblock Finiteness theorems and hyperbolic manifolds.
\newblock In {\em The Grothendieck Festschrift Volume III}, pages 163--178.
  Birkh\"{a}user, Boston, Basel, Berlin, 1990.

\bibitem[Qui08]{Quick08}
Gereon Quick.
\newblock {Profinite Homotopy Theory}.
\newblock {\em Documenta Mathematica}, 13:\ 585--612, 2008.

\bibitem[Qui10]{Quick10}
Gereon Quick.
\newblock {Continuous Group Actions on Profinite Spaces}.
\newblock {\em Journal of Pure and Applied Algbra}, 215(5):\ 1024--1039, 2010.

\bibitem[Qui12]{Quick12}
Gereon Quick.
\newblock {Some remarks on profinite completion of spaces}.
\newblock In {\em Galois-Teichm\"{u}ller theory and Arithmetic Geometry}, pages
  413--448. 2012.

\bibitem[Qui13]{Quick13}
Gereon Quick.
\newblock {Profinite $G$-Spectra}.
\newblock {\em Homology, Homotopy and Applications}, 15(1):\ 151--189, 2013.

\bibitem[Sch16]{JSchmidt16}
Johannes Schmidt.
\newblock {Sections, Homotopy Rational Points and Reductions of Curves}.
\newblock {\em arXiv: 1507.03831}, 2016.

\bibitem[SGA73]{SGA4.3}
{\em Th\'{e}orie de Topos et Cohomologie Etale des Sch\'{e}mas (SGA 4), Tome
  3}, volume 305 of {\em Lecture Notes in Mathematics}.
\newblock Springer-Verlag, Berlin-Heidelberg-New York, 1973.
\newblock Un s\'{e}minaire dirig{\'e} par M.\ Artin, A.\ Grothendieck, J.L.\
  Verdier avec la collaboration de P.\ Deligne, B.\ Saint-Donat.
  
\bibitem[SGA77]{SGA4.5}
{\em Cohomologie Etale (SGA 4 $\frac{1}{2}$)}, volume 569 of {\em Lecture Notes in Mathematics}.
\newblock Springer-Verlag, Berlin-Heidelberg-New York, 1977.
\newblock Un s\'{e}minaire dirig{\'e} par P.\ Deligne avec la collaboration de J.F.\ Boutot, A.\ Grothendieck, L.\ Illusie, J.L.\ Verdier.

\bibitem[SS16]{SchmidtStix16}
Alexander Schmidt and Jakob Stix.
\newblock {Anabelian geometry with \'{e}tale homotopy types}.
\newblock {\em To appear in Annals of Mathematics}, 2016.

\bibitem[Sti02]{Stix02}
Jakob Stix.
\newblock {\em Projective Anabelian Curves in Positive Characteristic and
  Descent Theory for Log-Etale Covers}, volume 354 of {\em Bonner Mathematische
  Schriften}.
\newblock 2002.

\bibitem[Sti06]{Stix06}
Jakob Stix.
\newblock {A General Seifert-Van Kampen Theorem for Algebraic Fundamental
  Groups}.
\newblock {\em Publications of the Research Institute for Mathematical Sciences
  Kyoto University}, 42(3):\ 763--786, 2006.

\bibitem[Sti10]{Stix10}
Jakob Stix.
\newblock {On the period-index problem in light of the section conjecture}.
\newblock {\em American Journal of Mathematics}, 132(1):\ 157--180, 2010.

\bibitem[Sti12]{Stix12}
Jakob Stix.
\newblock {\em Rational Points and Arithmetic of Fundamental Groups, Evidence
  for the Section Conjecture}, volume 2054 of {\em Lecture Notes in
  Mathematics}.
\newblock Springer-Verlag, Berlin-Heidelberg-New York, 2012.



\end{thebibliography}
\end{document}